\documentclass[11]{amsart}
\usepackage{amscd,amssymb,mathabx}
\usepackage[all]{xy}
\usepackage{graphicx,calrsfs}
\usepackage[colorlinks,plainpages,backref,urlcolor=blue]{hyperref}
\newtheorem{theorem}{Theorem}[section]
\newtheorem{corollary}[theorem]{Corollary}
\newtheorem{lemma}[theorem]{Lemma}
\newtheorem{prop}[theorem]{Proposition}

\theoremstyle{definition}
\newtheorem{definition}[theorem]{Definition}
\newtheorem{example}[theorem]{Example}
\newtheorem{remark}[theorem]{Remark}
\newtheorem*{ack}{Acknowledgments}

\newcommand{\Z}{\mathbb{Z}}
\newcommand{\Q}{\mathbb{Q}}

\newcommand{\C}{\mathbb{C}}
\renewcommand{\L}{\mathbf{L}}
\newcommand{\TT}{\mathbb{T}}
\newcommand{\PP}{\mathbb{P}}

\DeclareMathAlphabet{\pazocal}{OMS}{zplm}{m}{n}

\newcommand{\XX}{{\pazocal X}}

\newcommand{\CC}{{\pazocal{C}}}

\newcommand{\RR}{{\mathcal R}}
\newcommand{\VV}{{\mathcal V}}

\newcommand{\F}{{\mathcal{F}}}
\newcommand{\cC}{{\mathcal{C}}}

\newcommand{\E}{\mathsf{E}}
\newcommand{\V}{\mathsf{V}}
\newcommand{\W}{\mathsf{W}}

\newcommand{\g}{{\mathfrak{g}}}
\newcommand{\f}{{\mathfrak{f}}}
\newcommand{\h}{{\mathfrak{h}}}
\newcommand{\kk}{{\mathfrak{k}}}
\newcommand{\gl}{{\mathfrak{gl}}}
\renewcommand{\sl}{{\mathfrak{sl}}}
\newcommand{\sol}{{\mathfrak{sol}}}
\newcommand{\m}{{\mathfrak{m}}}
\newcommand{\fa}{{\mathfrak{a}}}

\newcommand{\cart}{{\mathfrak{t}}}
\newcommand{\n}{{\mathfrak{n}}}
\newcommand{\heis}{{\mathfrak{heis}_3}}

\newcommand{\G}{{\Gamma}}

\DeclareMathOperator{\rank}{rank}
\DeclareMathOperator{\gr}{gr}
\DeclareMathOperator{\im}{im}

\DeclareMathOperator{\id}{id}
\DeclareMathOperator{\ab}{{ab}}

\DeclareMathOperator{\GL}{GL}
\DeclareMathOperator{\SL}{SL}

\DeclareMathOperator{\Sp}{Sp}
\DeclareMathOperator{\PSL}{PSL}
\DeclareMathOperator{\Hom}{{Hom}}
\DeclareMathOperator{\spn}{span}

\DeclareMathOperator{\Der}{{Der}}
\DeclareMathOperator{\ad}{ad}

\DeclareMathOperator{\Lie}{Lie}
\DeclareMathOperator{\TC}{TC}
\DeclareMathOperator{\supp}{supp}

\DeclareMathOperator{\ideal}{ideal}
\DeclareMathOperator{\init}{in}

\newcommand{\same}{\Longleftrightarrow}
\newcommand{\surj}{\twoheadrightarrow}
\newcommand{\inj}{\hookrightarrow}

\newcommand{\isom}{\xrightarrow{\,\simeq\,}}

\def\dot{\mathchar"013A}
\newcommand{\hdot}{{\raise1pt\hbox to0.35em{\Huge $\dot$}}}
\newcommand{\cdga}{\ensuremath{\mathsf{cdga}}}
\newcommand{\raag}{\ensuremath{\mathsf{raag}}}
\newcommand{\DG}{\ensuremath{\mathsf{dg}}}
\newcommand{\dgla}{\ensuremath{\mathsf{dgla}}}

\begin{document}

\title[Flat connections and resonance varieties]{%
Flat connections and resonance varieties: from rank one to higher ranks}

\author[A.D.~Macinic]{Daniela Anca~M\u acinic$^1$}
\address{Simion Stoilow Institute of Mathematics,
P.O. Box 1-764, RO-014700 Bucharest, Romania}
\email{Anca.Macinic@imar.ro}
\thanks{$^1$ Supported by a grant of the Romanian National Authority for Scientific
Research, CNCS  UEFISCDI, project number PN-II-RU-PD-2011-3-0149}

\author[S.~Papadima]{\c Stefan Papadima$^2$}
\address{Simion Stoilow Institute of Mathematics,
P.O. Box 1-764, RO-014700 Bucharest, Romania}
\email{Stefan.Papadima@imar.ro}
\thanks{$^2$Partially supported by 
the Romanian Ministry of 
National Education, CNCS-UEFISCDI, grant PNII-ID-PCE-2012-4-0156}

\author[C.R.~Popescu]{Clement Radu Popescu$^3$}
\address{Simion Stoilow Institute of Mathematics, 
P.O. Box 1-764, RO-014700 Bucharest, Romania}
\email{Radu.Popescu@imar.ro}
\thanks{$^3$Supported by a grant of the Romanian
National Authority for Scientific Research, CNCS-UEFISCDI, 
project number PN-II-RU-TE-2012-3-0492}

\author[A.I.~Suciu]{Alexander~I.~Suciu$^4$}
\address{Department of Mathematics,
Northeastern University,
Boston, MA 02115, USA}
\email{a.suciu@neu.edu}
\thanks{$^4$Partially supported by
NSF grant DMS--1010298 and NSA grant H98230-13-1-0225}



\begin{abstract}
Given a finitely-generated group $\pi$ and a linear algebraic group $G$, 
the representation variety $\Hom(\pi,G)$  has a natural filtration by the 
characteristic varieties associated to a rational representation of $G$.  
Its algebraic counterpart, the space of $\g$-valued flat connections 
on a commutative, differential graded algebra $(A,d)$ admits a filtration by the 
resonance varieties associated to a representation of $\g$.  We establish 
here a number of results concerning the structure and qualitative properties 
of these embedded resonance varieties, with particular attention to the case when the 
rank $1$ resonance variety decomposes as a finite union of linear subspaces. 
The general theory is illustrated in detail in the case when $\pi$ is either an 
Artin group, or the fundamental group of a smooth, quasi-projective variety.
\end{abstract}

\maketitle
\setcounter{tocdepth}{1}
\tableofcontents
\newpage
\section{Introduction}
\label{sect:intro}

\subsection{Representation varieties and flat connections}
\label{subsec:intro reps}

One way to understand a finitely generated group $\pi$ is to pick 
a complex linear algebraic group $G$, and look at the resulting 
{\em representation variety}, $\Hom(\pi,G)$. In the rank $1$ case, one 
finds the character group $\TT(\pi)=\Hom(\pi,\C^{\times})$, which 
is simply the Pontryagin dual of the abelianization of $\pi$. 
In general, though, the $G$-representation variety is highly 
complicated, yet carries intricate information on the given 
group $\pi$, see for instance \cite{CS, GM, KM}.
The Kapovich--Millson universality theorem from \cite{KM} describes 
the complexity of representation varieties in a precise way: the analytic 
germs away from the trivial representation $1\in \Hom(\pi, \PSL_2)$, 
where $\pi$ runs through the family of Artin groups, are as complicated 
as arbitrary germs of affine varieties defined over $\Q$.

To make things simpler, it pays off to look at a closely related 
algebraic analogue of $\Hom(\pi,G)$.  Given a commutative, 
differential graded algebra (for short, a $\cdga$) $A=(A^\hdot, d)$ 
and a Lie algebra $\g$, the tensor product $A\otimes \g$ supports 
a graded, differential Lie algebra structure, with Lie bracket 
$[\alpha\otimes x, \beta\otimes y]= \alpha \beta \otimes [x,y]$
and differential $\partial (\alpha\otimes x) = d \alpha \otimes x$.
A $\g$-valued {\em flat connection} on $A$ is a tensor $\omega\in A^1\otimes \g$ 
that satisfies the Maurer--Cartan equation, 
$\partial\omega + \tfrac{1}{2} [\omega,\omega] = 0$.

If both $A^1$ and $\g$ are finite-dimensional, the set $\F(A,\g)$ 
of $\g$-valued flat connections on $A$ is a closed subvariety of 
the affine space $A^1\otimes \g$, containing the natural basepoint $0$.  
Moreover, if $A$ is a $\cdga$ model for a classifying space for $\pi$, 
and $\g$ is the Lie algebra of $G$, then $\F(A,\g)$ may be viewed 
as an infinitesimal version of $\Hom(\pi,G)$.
In the rank $1$ case, things are again very simple: if $A$ is 
connected (i.e., $A^0= \C \cdot 1$), then $\F(A,\C)= H^1(A)$. 
In higher ranks, we pay special attention to the Lie algebra 
$\sl_2$ and its standard Borel subalgebra $\sol_2$, in which 
case the variety of flat connections is more amenable to 
concrete descriptions.  

\subsection{Cohomology jump loci}
\label{subsec:intro cjl}

Both the representation variety and the variety of flat 
connections admit natural stratifications according to 
the dimensions of the relevant cohomology groups. 

In the first case, let $X$ be a connected CW-complex 
with $\pi_1(X)=\pi$, and assume $X$ has finite $q$-type 
(i.e., $X$ has finite $q$-skeleton), for some $q\ge 1$. 
For each rational representation 
$\iota\colon G\to \GL(V)$, and each integer $i\le q$, the sets 
\begin{equation}
\label{eq:intro cv}
\VV^i_r(X, \iota)=\{\rho \in \Hom(\pi, G) \mid
\dim_{\C} H^i(X, {}_{\iota\rho}V)\ge r\}
\end{equation}
are Zariski-closed in the representation variety.  Together with 
the ambient space, these {\em characteristic varieties}\/ 
(in degree $i$ and depth $r$) are homotopy-type invariants of $X$,  
encoding a plethora of information on the topology of the space $X$ 
and its covers, and reflecting certain geometric structures that $X$ 
may possess.  These varieties (also known as the cohomology 
jumping loci, or  the Green--Lazarsfeld sets in the 
algebro-geometric setting) have been the subject of much 
investigation, see for instance \cite{Ar, DP, DPS, PS-plms, Su-imrn}.  

In the second case, let $A$ be a connected $\cdga$, 
and assume $A$ has finite $q$-type (i.e., $A^{\le q}$ 
is finite-dimensional).  Given a 
representation $\theta\colon \g\to \gl(V)$, one may turn 
the tensor product $A\otimes V$ into a cochain complex, 
with differential $d_{\omega}=d\otimes \id_V + \ad_{\omega}$ 
depending on $\omega$ and $\theta$, see \eqref{eq:adw}. 
If both $\g$ and $V$ are finite-dimensional, the sets 
\begin{equation}
\label{eq:intro res}
\RR^i_r(A, \theta)= \{\omega \in \F(A,\g)
\mid \dim_{\C} H^i(A \otimes V, d_{\omega}) \ge  r\}
\end{equation}
are Zariski-closed in the parameter space for 
flat connections.  These {\em resonance varieties}\/ are 
more computable, and carry information in their own right 
about a space $X$ modeled by $A$.  

The simplest situation is that of the rank $1$ resonance varieties 
$\RR^i_r(A)$, when $\g=\C$ and $\theta=\id_{\C}$.  In the particular 
case when $A=H^{\hdot}(X,\C)$ is the cohomology algebra of a space $X$, 
endowed with the zero differential, these are the much-studied 
resonance varieties of $X$, denoted by $\RR^i_r(X)$.

\subsection{Overview of results}
\label{intro:overview}

The resonance and characteristic varieties are intimately related, 
at least around the origin.
More precisely, let $X$ be a space of finite $q$-type, and suppose 
Sullivan's de Rham model $\Omega^{\hdot}(X)$ has the same $q$-type 
as a $\cdga$ $A$ of finite $q$-type.  Let $\iota\colon G\to \GL(V)$ be a 
rational representation, and let $\theta\colon \g \to \gl(V)$
be its tangent representation.  

A foundational result from \cite{DP}, 
extending work from \cite{DPS}, states the following: There is an analytic 
isomorphism between the germ of $\F(A,\g)$ at the trivial connection $0$ 
and the germ of $\Hom(\pi,G)$ at the trivial representation $1$, which 
restricts to analytic isomorphisms 
$\RR^i_r(A,\theta)_{(0)} \cong \VV^i_r(X,\iota)_{(1)}$,
for all $i\le q$ and $r\ge 0$.

In this note, we study the {\em global}\/ structure of the resonance varieties 
sitting inside the variety of $\g$-valued flat connections associated to a 
$\cdga$ $A$. When $A$ models $X$, and both objects satisfy appropriate 
finiteness assumptions, this approach provides valuable topological 
information on the space $X$.  We obtain both general results, and 
explicit global descriptions for classes of $\cdga$'s related to important 
examples from group theory and algebraic geometry.

Our approach is motivated by the following natural question, 
which provides the main theme for the paper:    
Given a $\cdga$ $A$, how much information on its higher rank, 
embedded jump loci, $(\F(A,\g), \RR^i_r(A, \theta))$, can be 
extracted from the knowledge of its rank $1$ embedded jump 
loci, $(\F(A,\C), \RR^i_r(A))$?

\subsection{A tangent cone formula}
\label{subsec:intro tc}

It is easy to see that the construction $\F(A,\g)$ is a bifunctor. We devote 
a large part of \S\ref{sect:prel} to a general analysis of the functoriality 
properties of  resonance varieties, which are much more delicate. 
This leads us to the following, rather general ``tangent cone"-type result.

\begin{theorem}  
\label{thm:intro1}
Assume $A$ has the same $q$-type as $\Omega^{\hdot}(X)$, where 
$A$ and $X$ are of finite $q$-type. Then,
for all $i\le q$ and all $r\ge 0$, the tangent cone at $0$ to the 
resonance variety $\RR^i_r(A)$ is contained in $\RR^i_r(H^{\hdot}(A))$.

If, moreover, $A$ is defined over $\Q$, the
identification $H^1(\Omega^{\hdot}(X)) \cong H^1(A)$
preserves $\Q$-structures, and $A$ has positive weights, 
then $\RR^i_r(A)\subseteq \RR^i_r(H^{\hdot}(A))$, in the same range.
\end{theorem}

The first part of the theorem can be viewed as a rational homotopy 
analogue of a result of Libgober \cite{Li}, which says the following:  
if $X$ is a connected, finite-type CW-complex, then,  for all $i$ and $r$, 
the tangent cone at $1$ to the characteristic variety $\VV^i_r(X)$ is contained 
in $\RR^i_r(X)$.

Examples of spaces $X$ of finite $q$-type equipped with a finite $q$-type 
$\cdga$ $A$ having the same $q$-type as $\Omega^{\hdot}(X)$ abound.  
For instance, if $X$ has finite $q$-type and is $q$-formal, we may take 
$A$ to be the cohomology algebra $H^{\hdot}(X,\C)$, endowed with the 
zero differential.  When $X$ is a compact K\"{a}hler manifold, and thus 
$X$ is formal, these properties hold even for $q=\infty$. 
Another important class of examples (where  $q=\infty$) consists of 
irreducible, smooth, quasi-projective varieties (for short, quasi-projective 
manifolds), which come equipped with one of Morgan's Gysin models \cite{Mo}.  
In both the $q$-formal and the quasi-projective cases, the additional 
conditions from Theorem \ref{thm:intro1} also hold. 
Finally, if $X=K(\pi,1)$ is a classifying space for a finitely-generated 
nilpotent group $\pi$, we may take $A$ to be the cochain 
algebra of the Malcev--Lie algebra associated to $\pi$.
We refer to \cite{DP} for more examples and details related to this topic.

\subsection{Essentially rank $1$ flat connections}
\label{subsec:intro rk1}

Returning now to the general setup, assume $A$ is $1$-finite and $\g$ is 
finite-dimensional.  Consider the closed subvariety through the origin, 
$\F^1(A,\g)\subseteq \F(A,\g)$, consisting of all tensors of the form 
$\eta \otimes g$ with $d \eta=0$.  
Given a finite-dimensional representation $\theta\colon \g\to \gl (V)$, 
we let $\Pi(A,\theta)\subseteq \F^1(A,\g)$ be the subvariety 
of all such tensors which also satisfy $\det(\theta(g))=0$. Note that 
$\F^1(A,\g)$ depends only on the vector spaces $H^1(A)$ and $\g$, while  
$\Pi(A,\theta)$ depends only on $H^1(A)$ and the representation $\theta$.
Both these varieties are easy to compute in any concrete situation.  
For instance, in the rank $1$ case, $\F^1(A,\C)= \F(A,\C)$ and $\Pi(A,\theta)= \{ 0\}$.
Using these notions, we show in \S\ref{subsec:rk1 res} that the trace of higher 
rank depth $1$ resonance on $\F^1(A,\g)$ is determined by rank $1$ resonance, 
in the following precise way.
 
\begin{theorem}
\label{thm:intro2}
Let $\omega=\eta\otimes g$ be an arbitrary element of
$\F^1(A, \g)$.  Then $\omega$ belongs to $\RR^k_1(A,\theta)$
if and only if there is an eigenvalue $\lambda$ of $\theta(g)$ such that
$\lambda \eta$ belongs to $\RR^k_1(A)$. Moreover,
\[
\Pi(A,\theta)\subseteq \bigcap_{k: H^k(A)\ne 0} \RR^k_1(A, \theta).
\]
\end{theorem}

As a first application of this general result, we determine in \S\ref{subsec:nilp} the
germs at $1$ of embedded depth $1$ characteristic varieties, for a finitely-generated 
nilpotent group $\pi$, in the case when $G$ is either semisimple of rank one, or 
one of its Borel subgroups.  Given a rational representation $\iota\colon G \to \GL(V)$, 
with tangent map $\theta$, we show that $\F(A,\g)= \F^1(A,\g)$, where $A=\CC^\hdot(\n)$ 
is the Chevalley--Eilenberg cochain algebra of the Malcev--Lie algebra $\n$ associated 
to $\pi$. Moreover,
\begin{equation}
\label{eq:intro res nilp}
\RR_1^k (\CC^\hdot(\n), \theta)=\Pi(\CC^\hdot(\n), \theta)
\end{equation}
if $k \leq \dim \n$ and is empty otherwise.

\subsection{Holonomy Lie algebras}
\label{subsec:intro hlie}

In \S\ref{sect:holo lie}, we consider the Lie analogue of resonance varieties, 
expanding on the approach from \cite{DPS}.  Given a Lie algebra $\h$, 
and a representation $\theta\colon \g\to \gl(V)$ of another Lie algebra, 
the corresponding resonance varieties are defined as 
\begin{equation}
\label{eq:intro lie res}
\RR^i_r(\h, \theta)= \{\rho \in \Hom_{\Lie}(\h,\g)
\mid \dim_{\C} H^i(\h,  V_{\theta\circ\rho}) \ge  r\}.
\end{equation}

We relate this notion of resonance to the previous one, under the 
assumption that both $\g$ and $V$ have finite dimension. 
Given a connected $\cdga$ $A$ with $A^1$ finite-dimensional,
we define in a functorial way the {\em holonomy Lie algebra}\/ $\h(A)$,  
as a certain quotient of the free Lie algebra on the dual vector space 
$A_1=(A^1)^*$. We then prove the following result.

\begin{theorem}
\label{thm=lieresintro}
The natural linear isomorphism $\psi\colon A^1\otimes \g \cong  \Hom (A_1,\g)$ 
restricts to an isomorphism
$\psi\colon \F(A,\g) \cong  \Hom_{\Lie} (\h(A), \g)$, and induces  isomorphisms
$\RR^i_r(A,\theta) \cong \RR^i_r (\h(A), \theta)$, for $i\le 1$ and $r\ge 0$. 

Furthermore, if $A= \CC^\hdot(\f)$, where $\f$ is a Lie algebra of finite dimension, 
then $\h(A)= \f$, and $\psi$ induces  isomorphisms
$\RR^i_r(A,\theta) \cong \RR^i_r (\f, \theta)$, for all $i,r \ge 0$. 
\end{theorem}

We also show that the non-triviality of the depth $1$, degree $1$ 
resonance varieties is detected in the rank $1$ case by the Lie 
algebra $\sol_2$, and by the $3$-dimensional Heisenberg Lie 
algebra $\heis$.  More precisely, $\RR^1_1(A) \nsubseteq \{0\}$ 
if and only if $\Hom_{\Lie}(\h(A), \sol_2)$ contains a surjection, 
while $\RR^1_1(H^\hdot (A)) \nsubseteq \{0\}$ if and only if
$\Hom_{\Lie}(\h(A), \heis)$ contains a surjection.

\subsection{Linear resonance}
\label{subsec:intro linear}

Starting in \S\ref{sect:linres}, and for the rest of the paper, we focus our attention 
on the class of \cdga's $A$ for which the rank $1$ resonance variety $\RR^1_1(A)$ 
decomposes as a finite union of linear subspaces. From this point on, we 
assume the bare minimum in terms of finiteness conditions that are needed 
for our approach to work well:  $A$ will be a connected, $1$-finite $\cdga$, 
and $\theta$ will be a finite-dimensional representation of a Lie algebra $\g$ 
of finite dimension. 

It turns out that, under some additional hypotheses on $A$ and $\g$, 
the structure of both the parameter space $\F(A,\g)$ and of the higher-rank 
resonance varieties $\RR_1^1(A, \theta)$ is determined by the 
decomposition of the rank $1$ resonance variety into {\em linear}\/ 
irreducible components.  More precisely, we prove the following 
structural result. 

\begin{theorem}
\label{thm:intro3}
Suppose $\RR^1_1(A)=\bigcup_{C\in \cC} C$, a finite union 
of linear subspaces. For each $C\in \cC$, let $A_C$ 
denote the sub-$\cdga$ of the truncation $A^{\le 2}$ 
defined by $A^1_C=C$ and $A^2_C=A^2$. 
Then, for any Lie algebra $\g$, 
\begin{equation}
\label{eq:i1}
\F(A, \g) \supseteq \F^1(A, \g) \cup \bigcup_{0\ne C\in\cC}\F(A_C, \g),
\end{equation}
where each $\F(A_C, \g)$ is Zariski-closed in $\F(A, \g)$.   Moreover, 
if $A$ has zero differential, $A^1$ is non-zero, and 
$\g =\sl_2$ or $\sol_2$, then \eqref{eq:i1} holds as 
an equality, and, for any $\theta$, 
\begin{equation}
\label{eq:i2}
\RR_1^1(A, \theta)=\Pi(A, \theta) \cup \bigcup_{0\ne C\in \cC} \F(A_C, \g).
\end{equation}
\end{theorem}

For instance, consider the case when $X$ is a $1$-finite, $1$-formal space, 
and $A=H^{\hdot}(X,\C)$ with differential $d=0$. Then, as shown in \cite{DPS}, 
the resonance variety $\RR^1_1(A)$ is linear. By a result from \cite{DP}, Gysin 
models of quasi-projective manifolds also have the linearity property for 
$\RR^1_1(A)$, but the differential in this case may well be non-zero.

\subsection{Artin groups}
\label{subsec:intro artin}

In \S\ref{sect:raag}, we analyze the Kapovich--Millson family of representation 
varieties, near the origin.  Given a finite simplicial graph $\G$, with edges 
labeled by integers greater than $1$, there is an associated Artin group $\pi_{\G}$;  
if all labels are equal to $2$, the group $\pi_\G$ is called a right-angled 
Artin group.  According to \cite{KM}, every Artin group $\pi_{\G}$ has a 
classifying space $X_{\G}$ which is $1$-finite and $1$-formal. 
For the purpose of computing the embedded degree $1$ resonance 
varieties, we may model this space by the $\cdga$ 
$A_{\G}=(H^{\hdot}(\pi_{\G}, \C), d=0)$, to which Theorem \ref{thm:intro3} applies. 

We use holonomy Lie algebras, as well as results from \cite{PS-artin}
and \cite{DPS}  to describe explicitly the global irreducible 
decompositions for $\F(A_\G, \g)$ and $\RR_1^1(A_\G, \theta)$, in terms 
of the ``odd contraction" of the labeled graph $\G$ and a representation 
$\theta$ of $\g=\sl_2$.  For instance, if $\pi_\G$ is a right-angled Artin group, then 
\begin{equation}
\label{eq:fraag}
\F(A_\G, \sl_2)= \bigcup_{\W\subseteq \V} S_{\W}
\end{equation}
where $\W$ runs through the subsets of the vertex set of $\G$, 
maximal with respect to an order $\preccurlyeq$ defined in terms 
of the connected components of the induced subgraph $\G_{\W}$, 
while $S_{\W}$ is a certain combinatorially defined, closed subvariety of 
$\C^{\W}\otimes \sl_2$, in fact, a product of cones 
over varieties of the form $\PP^n \times \PP^2$, see \eqref{eq:prodcone}.

Taking germs at $0$, these results give a concrete way to compute the germs 
at the origin for both $\Hom(\pi_{\G}, \PSL_2)$ and  $\VV^1_1(X_{\G},\iota)$,  
where $\iota$ is an arbitrary rational representation of $\PSL_2$. Our explicit 
description of the singularity germ $\Hom(\pi_{\G}, \PSL_2)_{(1)}$  complements 
the qualitative quadraticity restriction on this germ given by Kapovich 
and Millson in \cite{KM}.

For an arbitrary Lie algebra $\g$, we show that the variety  
$\F(A_{\G}, \g)$ contains the union of the subvarieties $S_{\W}$, 
defined in a similar manner as above. If $\g$ is semisimple and different 
from $\sl_2$, though, we show by example that this containment 
can be strict.

\subsection{Quasi-projective manifolds}
\label{subsec:intro qp}

A basic object of study in algebraic geometry is the class of quasi-projective 
varieties.  In \S\ref{sect:qp}, we apply our machinery to quasi-projective manifolds. 

Let $X$ be such a manifold, and assume $b_1(X)>0$. In this case, the 
irreducible decomposition of $\RR_1^1(A)$, for a suitable Gysin model 
$A$ of $X$, can be described in geometric terms. The irreducible 
components are all linear, and they are indexed 
by a finite list, denoted $\mathcal{E}_X$, of equivalence classes of 
regular ``admissible" maps $f\colon X\to S$, 
where the quasi-projective manifolds $S$ are $1$-dimensional, and 
have negative Euler characteristic. More precisely, it follows from \cite{DP} 
that 
\begin{equation}
\label{eq:res11}
\RR_1^1(A)= \{ 0\} \cup \bigcup_{f\in \mathcal{E}_X} f^* (H^1(A(S))),
\end{equation}
where $A(S)$ is Morgan's Gysin model of $S$. If $X$ is also 
$1$-formal, it follows from \cite{DPS} that the same decomposition formula 
holds with Gysin models replaced by cohomology rings, 
endowed with the zero differential.

In the $1$-formal situation, we use again Theorem \ref{thm:intro3} to find 
explicit (global) irreducible decompositions for $\F(H^{\hdot}(X,\C), \g)$ 
and $\RR_1^1(H^{\hdot}(X,\C), \theta)$ in the case when $\g=\sl_2$, 
solely in terms of the set $\mathcal{E}_X$ and of the representation 
$\theta$. In the (much more delicate) general case, we are able to 
show that $\F(A, \g)= \F^1(A, \g)$ and $\RR_1^1(A, \theta) =\Pi (A, \theta)$, 
for an arbitrary Gysin model $A$ of $X$, and for a representation $\theta$ of
$\g=\sl_2$ or $\sol_2$, provided that $\RR_1^1(X)= \{ 0\}$. 

In \cite{PS-beta}, we apply this machinery to the case when $X$ is the 
complement of a complex hyperplane arrangement. Exploiting the known 
decomposition into irreducible components of $\RR^1_1(X)$, and the 
fact that $X$ is formal, we determine the irreducible decomposition 
of $\F(H^{\hdot}(X,\C),\sl_2)$, and then relate this decomposition 
to the algebraic monodromy of the Milnor fibration of the arrangement.

\subsection{Notations and conventions}
\label{subsec:notation}

Throughout, we work over $\C$. Affine varieties and analytic germs will 
always be reduced.  For a point $v$ on an affine variety  $\mathcal{V}$, 
we let $\mathcal{V}_{(v)}$ denote the associated germ.

We denote vector space duals by $V^*$. For a graded vector
space $A^{\hdot}$, we write $A_{\hdot}=(A^{\hdot})^*$.
Given two vector spaces, $V$ and $W$, we denote by
$\Hom^1(V,W)$ the subset of those linear maps from $V$ to $W$
whose image is at most $1$-dimensional. We also
let $\psi\colon V\otimes W\to \Hom (V^*,W)$ be the linear
map given by $\psi(v\otimes w)(u^*) =\langle v, u^* \rangle w$,
where $\langle \cdot, \cdot \rangle$ stands for the duality pairing. 
(This map is always injective, and it is an isomorphism when
$\dim V<\infty$.)

\section{Flat connections and resonance varieties}
\label{sect:prel}

We start by reviewing the two basic algebraic constructions in 
the title of this section, paying special attention to their functoriality 
properties, and to the rank $1$ case.

\subsection{Representation varieties and jump loci}
\label{subsec:charvars}

Let $X$ be a path-connected space, and let $\pi=\pi_1(X)$ be its fundamental group.
We say that $X$ is of finite $q$-type (or {\em $q$-finite}), for some integer
$q\ge 1$, if $X$ has the homotopy type of a CW-complex with finite $q$-skeleton.

Now let $G$ be a linear algebraic group. The set of group homomorphisms from 
$\pi$ to $G$, denoted $\Hom(\pi, G)$, is called the {\em $G$-representation variety}\/ 
of $\pi$.  This set comes with an obvious basepoint: the trivial representation, 
$1\in \Hom(\pi, G)$. If the space $X$ is $1$-finite, or, equivalently, the group 
$\pi$ is finitely generated, then $\Hom(\pi, G)$ has a natural affine structure.

Next, fix a rational representation $\iota\colon G \to \GL (V)$.
By definition, the {\em characteristic varieties}\/ (or, {\em cohomology
jump loci}) of $X$, in degree $i\ge 0$ and depth $r\ge 0$, and with
respect to the representation $\iota$, are the sets
\begin{equation}
\label{eq:defv}
\VV^i_r(X, \iota)=\{\rho \in \Hom(\pi, G) \mid
\dim_{\C} H^i(X, {}_{\iota\rho}V)\ge r\},
\end{equation}
where ${}_{\iota\rho}V$ denotes the left $\pi$-module defined
by the representation $\iota\circ \rho\colon \pi\to \GL(V)$.  If
the space $X$ is $q$-finite, the jump loci $\VV^i_r(X,\iota)$
are Zariski-closed subsets of the representation variety
$\Hom(\pi, G)$, for all $i\le q$ and $r\ge 0$.

The most basic (and best understood) case is the rank $1$ case,
which corresponds to the choices $G=\C^{\times}$, $V=\C$, and
$\iota \colon \C^{\times} \xrightarrow{=} \GL_1(\C)$. The
respective jump loci, denoted simply by $\VV^i_r(X)$,
are subsets of the character group $\TT(\pi):=\Hom(\pi,\C^{\times})$.

\begin{remark}
\label{rem:dual}
Let $\chi\colon \pi\to \GL(V)$ be a finite-dimensional
representation of the group $\pi=\pi_1(X)$, defining
a left $\pi$-module ${}_{\chi}V$, and let $V^*_{\chi}$
be the dual (right) $\pi$-module.  We then have a
natural duality isomorphism, $H^{\hdot}(X, {}_{\chi}V)
\isom H_{\hdot}(X, V^*_{\chi})^*$, see for instance Lemma 8.5
from preprint v.1 of \cite{DP}.  Hence,  when
defining the rank $1$ characteristic varieties
$\VV^i_r(X)$, cohomology may be replaced
by homology, as done, for instance, in
\cite{PS-plms, PS-specres}.
\end{remark}

\subsection{Flat connections}
\label{subsec:flat}
We switch now to a more algebraic (yet very much related)
framework.  Let $A=(A^{\hdot},d)$ be a commutative, differential
graded algebra (for short, a \cdga) over $\C$.  We say that $A$ is of
finite $q$-type  (or {\em $q$-finite}), for some $1\le q\le \infty$, if $A$ is connected 
(that is, $A^0$ is the $\C$-span of the unit $1$), and $A^{\le q}$ is finite-dimensional.

Next, let $\g$ be a Lie algebra over $\C$.
Given any $\cdga$ $A$, consider the tensor product $A\otimes \g$.
On this vector space, we may define a Lie bracket given by
$[\alpha\otimes x, \beta\otimes y]= \alpha \beta \otimes [x,y]$, 
and a differential given by $\partial (\alpha\otimes x) = d \alpha \otimes x$.
This construction produces a graded, differential Lie algebra
(for short, \dgla), denoted $A\otimes \g = (A^{\hdot}\otimes \g, \partial)$.
Moreover, the construction is functorial in both $A$ and $\g$.

The algebraic analogue of the $G$-representation variety $\Hom(\pi,G)$
is the set of {\em $\g$-valued flat connections}\/ on $A$.  By definition,
this is the set $\F(A,\g)$ of degree $1$ elements in $A\otimes \g$ that
satisfy the Maurer--Cartan equation,
\begin{equation}
\label{eq:flat}
\partial\omega + \tfrac{1}{2} [\omega,\omega] = 0 .
\end{equation}

A typical element in $A^1 \otimes \g$
is of the form $\omega =\sum_i \eta_i \otimes g_i$, with
$\eta_i\in A^1$ and $g_i\in \g$; the flatness condition
amounts to
\begin{equation}
\label{eq:flat coords}
\sum_{i} d\eta_i \otimes g_i +
\sum_{i<j} \eta_i \eta_j \otimes [g_i, g_j] =0.
\end{equation}

The set of flat connections depends functorially on both
$A$ and $\g$, and comes with an obvious basepoint:
the trivial connection, $0\in \F(A,\g)$.  Now suppose
that both $A^1$ and $\g$ are finite-dimensional.
It is then readily seen that $\F(A,\g)$ is a Zariski-closed
subset of the affine space $A^1\otimes \g$, and
thus carries a natural affine structure.

The simplest situation is the rank one case, for which $\g=\C$.
In this case, the space $\F(A,\C)$ may be identified with the
vector space $Z^1(A)=\{\alpha \in A^1\mid d\alpha=0\}$.
In particular, if $d=0$, then $\F(A,\C)=A^1$.

\begin{remark}
\label{rem:gm}
The terminology we use here comes from differential geometry.  
Indeed, if $\Omega^{\hdot}(X)$ is the algebra of differential forms 
on a connected, smooth manifold $X$, and $ \g$ is the Lie algebra of a Lie 
group $G$, then $\F(\Omega^{\hdot}(X), \g)$ may be identified 
with the space of flat connections on the principal (trivial) bundle 
$X \times G \to X$. When $G$ is a linear algebraic group and 
$\pi_1(X)$ is finitely generated, the functorial monodromy, 
$\F(\Omega^{\hdot}(X),\g) \to \Hom(\pi_1(X), G)$, determines 
the germ at $1$ of the representation variety.
For a detailed treatment of this topic, we refer to \cite{GM}.
\end{remark}

\subsection{Aomoto complex}
\label{subsec:aomoto}
Returning now to the general situation, let $V$ be a complex
vector space, and let $\theta \colon \g \to \gl (V)$ be a
morphism of Lie algebras, also known as a representation 
of the Lie algebra $\g$ in $V$, or a $\g$-module structure on $V$.  
For each flat connection $\omega\in \F(A,\g)$, we make the 
tensor product $A\otimes V$ into a cochain complex,
\begin{equation}
\label{eq:aomoto}
\xymatrixcolsep{22pt}
\xymatrix{(A\otimes V , d_{\omega})\colon  \
A^0 \otimes V \ar^(.65){d_{\omega}}[r] & A^1\otimes V
\ar^(.5){d_{\omega}}[r]
& A^2\otimes V   \ar^(.55){d_{\omega}}[r]& \cdots },
\end{equation}
using as differential the covariant derivative
$d_{\omega}=d\otimes \id_V + \ad_{\omega}$. Here, the adjoint
operator $ \ad_{\omega}$ is defined by the semi-direct product Lie 
algebra $V\rtimes_{\theta} \g$. Explicitly, if 
$\omega=\sum_i \eta_i \otimes g_i$, then
\begin{equation}
\label{eq:adw}
d_{\omega}(\alpha\otimes v) = d\alpha \otimes v +
\sum_{i} \eta_i  \alpha \otimes \theta(g_i)v,
\end{equation}
for all $\alpha\in A^i$ and $v\in V$.  It is readily checked that
the flatness condition on $\omega$ insures that $d_{\omega}^2=0$.

The cochain complex \eqref{eq:aomoto}, which we will call
the {\em Aomoto complex}, enjoys the following naturality property.
Given a morphism $\varphi\colon A\to A'$ of \cdga's, set
$\omega'=\varphi\otimes \id_{\g} (\omega)$; then
$\varphi\otimes \id_{V}\colon (A\otimes V , d_{\omega}) \to
(A'\otimes V , d_{\omega'})$ is a cochain map.

Note that $A^\hdot \otimes V$ is a free graded left $A^\hdot$-module.
Moreover, it is easily checked that the Aomoto complex is a $\DG$-module over the
\cdga~$(A^\hdot, d)$. In particular, $H^\hdot (A^\hdot \otimes V, d_{\omega})$
acquires a natural graded $H^\hdot (A)$-module structure.

The next lemma provides  a simple interpretation of the $0$-th
cohomology of the Aomoto complex.

\begin{lemma}
\label{lem:aom0}
Let $A$ a be a connected \cdga, and let $\theta\colon \g\to \gl(V)$
be a  representation.  Pick a  basis $\{\eta_i\}$
for $A^1$, and suppose
$\omega =\sum_i \eta_i \otimes g_i \in \F(A,\g)$ is a
flat connection. Then
\[
H^0(A\otimes V, d_{\omega})=\bigcap_i \ker (\theta(g_i)).
\]
\end{lemma}

\begin{proof}
From the definitions, we have that $H^0(A\otimes V, d_{\omega}) =
\ker (d_{\omega} \colon A^0 \otimes V \to A^1\otimes V)$, where
$d_{\omega}(1\otimes v) = \sum_i \eta_i \otimes \theta(g_i) v $.
The conclusion follows at once.
\end{proof}

\subsection{Resonance varieties}
\label{subsec:resvar}
More generally, we may compute the cohomology of the Aomoto
complex in any fixed degree, and record the set of flat connections for
which the dimension of this vector space jumps  above a prescribed
value.  This leads to defining the {\em resonance varieties}\/ of a $\cdga$ $A$,
with coefficients given by a representation $\theta\colon \g\to \gl(V)$ as
\begin{equation}
\label{eq:rra}
\RR^i_r(A, \theta)= \{\omega \in \F(A,\g)
\mid \dim_{\C} H^i(A \otimes V, d_{\omega}) \ge  r\}.
\end{equation}

If $A$ is $q$-finite, and both $\g$ and $V$ are finite-dimensional,
the sets $\RR^i_r(A, \theta)$ are Zariski-closed subsets of
$\F(A,\g)$, for all $i\le q$ and $r\ge 0$.

In the rank one case, we will simply write $\RR^i_r(A):=\RR^i_r(A,\id_{\C})$
for the corresponding resonance varieties, viewed as algebraic
subsets of $H^1(A) =Z^1(A)$.  If $d=0$, the varieties $\RR^i_r(A)$
are homogeneous subsets of $A^1$.  In general, though,
the varieties $\RR^i_r(A)$ are {\em not}\/ homogeneous.

The resonance varieties of a path-connected space $X$ are traditionally
defined as the rank-$1$ resonance varieties of its cohomology algebra
(endowed with the zero differential); that is, $\RR^i_r(X):=\RR^i_r(H^{\hdot}(X,\C))$.
More generally, we set $\RR^i_r(X, \theta):=\RR^i_r(H^{\hdot}(X,\C),\theta)$, 
for a representation $\theta\colon \g\to \gl(V)$. 

To avoid trivialities, we will assume from now on that $V\ne 0$,
in which case
\begin{equation}
\label{eq:basept}
0\in \RR^i_1(A,\theta) \Leftrightarrow
0\in \RR^i_1(A)\Leftrightarrow H^i(A)\ne 0.
\end{equation}

Here is a simple situation where the resonance variety $\RR^1_1(A, \theta)$
can be identified explicitly.

\begin{prop}
\label{prop:chi}
Let $A^\hdot$ be a $2$-finite $\cdga$, and let $\theta\colon \g\to \gl(V)$
be a finite-dimensional representation.   If $A^{>2}=0$ and $\chi(H^\hdot (A))<0$,
then $\RR^1_1(A, \theta)=\F(A,\g)$.
\end{prop}

\begin{proof}
Let $\omega \in \F(A, \g)$ be an arbitrary element, and suppose
$H^1(A \otimes V, d_{\omega})=0$. Then clearly
$\chi(H^\hdot(A \otimes V, d_{\omega})) \geq 0$. On the other hand,
\[
\chi(H^\hdot(A \otimes V, d_{\omega}))=\chi(A^\hdot) \cdot \dim V=
\chi (H^\hdot (A)) \cdot \dim V<0,
\]
a contradiction.  Thus, $\omega \in \RR^1_1(A, \theta)$, and we are done.
\end{proof}

\subsection{Functoriality of resonance}
\label{subsec:funct}

Let us now discuss the functoriality properties of the resonance
varieties $\RR^i_r(A, \theta)$ attached to a $\cdga$ $A$
and a representation $\theta\colon \g\to \gl(V)$.

First, let $f\colon \g\to \g'$ be a morphism of Lie algebras. Clearly, the
map $\bar{f}=\id_{A^1}\otimes f\colon A^1\otimes \g\to A^1\otimes \g'$
restricts to a map $\bar{f}\colon \F(A,\g) \to \F(A,\g')$.
Hence, if $\theta'\colon \g' \to \gl(V)$ is a representation such that
$\theta=\theta'\circ f$, $\omega\in A^1 \otimes\g$, and
$\omega'=\bar{f}(\omega)$, then $d_{\omega} = d_{\omega'}$ and
\begin{equation}
\label{eq:univ}
\RR^i_r(A, \theta) =\bar{f}^{-1} (\RR^i_r(A, \theta')),
\end{equation}
for all $i\ge 0$ and $r\ge 0$.

\begin{remark}
\label{rem:univ}
By taking $f=\theta$ in formula \eqref{eq:univ}, we see that,
from the point of view of the resonance varieties of $A$, the case when
$\theta'$ is the identity of $\gl(V)$ determines all others; that is to say,
$\RR^i_r(A, \theta) =\bar\theta^{-1}(\RR^i_r(A, \id_{\gl(V)}))$.
\end{remark}

Next, let $\varphi\colon A\to A'$ be a morphism of \cdga's.
The map $\varphi\otimes \id_{\g}\colon A^1\otimes \g \to
A'^1\otimes \g$ restricts to a map $\bar\varphi\colon \F(A,\g) \to
\F(A',\g)$.  The next lemma shows that the resonance varieties
of $A$ and $A'$ agree, roughly in the range for which $\varphi$
is an isomorphism.

\begin{lemma}
\label{lem:functr}
Suppose $\varphi\colon A^{\hdot}\to A'^{\hdot}$  is an isomorphism
up to degree $q$, and a monomorphism in degree $q+1$, for some
$q\ge 0$.
\begin{enumerate}
\item\label{f1}
If $q\ge 1$, the map $\bar\varphi$ is an isomorphism which identifies
$\RR^i_r(A,\theta)$ with $\RR^i_r(A',\theta)$ for each $i\le q$,
and sends $\RR^{q+1}_r(A,\theta)$ into $\RR^{q+1}_r(A',\theta)$,
for all $r\ge 0$.
\item\label{f2}
If $q= 0$, the map $\bar\varphi$ is an embedding which identifies
$\RR^0_r(A,\theta)$ with $\RR^0_r(A',\theta)\cap \F(A,\g)$,
and sends $\RR^1_r(A,\theta)$ into $\RR^1_r(A',\theta)$,
for all $r\ge 0$.
\end{enumerate}
\end{lemma}

\begin{proof}
The fact that the map $\bar\varphi\colon \F(A,\g) \to
\F(A',\g)$ is an isomorphism, respectively, an embedding follows
straight from the definitions.

Let $\omega\in \F(A,\g)$, and set $\omega'=\bar\varphi(\omega)$.
The map $\varphi$ then defines a map of cochain complexes,
$\tilde\varphi:=\varphi \otimes \id_{V}\colon (A\otimes V,d_{\omega}) \to
(A'\otimes V,d_{\omega'})$. It is readily seen that our hypothesis
on $\varphi$ is inherited by the map $\tilde\varphi$,
as well as by the induced homomorphism  in cohomology,
$\tilde\varphi_*\colon H^{\hdot}(A\otimes V,d_{\omega}) \to
H^{\hdot}(A'\otimes V,d_{\omega'})$.  All the assertions on the
resonance varieties now follow from this observation.
\end{proof}

If we replace in the hypothesis of the above lemma the map 
$\varphi\colon A\to A'$ by the induced homomorphism 
$\varphi_*\colon H^{\hdot}(A)\to H^{\hdot}(A')$, the conclusions 
may no longer follow.  We illustrate this assertion with a simple 
example.

\begin{example}
\label{ex:sol2}
Let $A'^{\hdot}=\bigwedge^{\hdot} (h^*,x^*)$
be the exterior algebra with generators in degree $1$ and with
differential given by $d h^*=0$ and $d x^*= x^*\wedge h^*$ (as 
we shall see later, this is the cochain $\cdga$ of the solvable, $2$-dimensional
Lie algebra $\sol_2$), and let $\varphi\colon A\hookrightarrow A'$
be the inclusion of the sub-$\cdga$ $A^{\hdot}=\bigwedge^{\hdot} (h^*)$.
We then have $H^1(A)=H^1(A')=\C$, and $\varphi_*$ is an isomorphism
in all degrees.  Nevertheless, $\RR^1_1(A)=\{0\}$, while  $\RR^1_1(A')=\{0,1\}$.
\end{example}

\begin{remark}
\label{rem:noth}
The behavior illustrated in the previous example is in marked contrast 
with the local case. Indeed, assume $\varphi_*\colon H^{\hdot}(A)\to H^{\hdot}(A')$ 
is an isomorphism up to degree $q$ and a monomorphism in degree $q+1$. 
Let $\fa$ be an Artinian local algebra, with maximal ideal $\m$. 
For $\omega\in \F(A,\g \otimes \m)$, set 
$\omega'=\varphi \otimes \id_{\g} \otimes \id_{\fa} (\omega)$. 
Then, as shown in \cite[Theorem 3.7]{DP}, the induced homomorphism, 
$\tilde\varphi_*\colon H^{\hdot}(A\otimes V \otimes \fa,d_{\omega}) \to
H^{\hdot}(A'\otimes V \otimes \fa,d_{\omega'})$, 
inherits the properties of $\varphi_*$. 

Returning now to the $\cdga$ map $\varphi\colon A\hookrightarrow A'$ 
from Example \ref{ex:sol2}, take $(\fa, \m)=(\C,0)$, and consider the
non-local flat connection $\omega=1\in \F(A, \C\otimes \fa)$. We then 
have $H^{1}(A\otimes \C \otimes \fa,d_{\omega})=0$, while 
$H^{1}(A'\otimes \C \otimes \fa,d_{\omega'}) \ne 0$.
\end{remark}

\subsection{Germs of jump loci}
\label{subsec:germs}
The link between the characteristic and resonance varieties
of a space is provided by a construction due to Dennis Sullivan \cite{Su}.

Given a path-connected space $X$, let $\Omega^{\hdot}(X)$ be
Sullivan's de Rham model of $X$, a $\cdga$ which mimics the algebra 
of forms on a smooth manifold, and for which the de Rham theorem holds.
We say that two \cdga's $A$ and $B$ have the {\em same $q$-type},
for some $q\ge 1$ (written $A\simeq_q B$) if they can be connected
by a zig-zag of $\cdga$ maps inducing isomorphisms in cohomology
in degree up to $q$, and a monomorphism in cohomology in degree
$q+1$.  We then have the following foundational result of Dimca
and Papadima \cite{DP} (see also \cite{DPS} for the case when
$q=1$ and $X$ is $1$-formal).

\begin{theorem}[\cite{DP}]
\label{thm:b1}
Let $X$ be a space of finite $q$-type, and assume $\Omega^{\hdot}(X)$
has the same $q$-type as a $\cdga$ $A$ of finite $q$-type.
Let $\iota\colon G\to \GL(V)$ be a rational representation of
a linear algebraic group $G$, and let $\theta\colon \g \to \gl(V)$
be its tangent representation.

There is then an analytic isomorphism
of germs, $e\colon \F(A,\g)_{(0)} \isom \Hom(\pi,G)_{(1)}$,
which restricts to analytic isomorphisms
$e\colon \RR^i_r(A,\theta)_{(0)} \isom \VV^i_r(X,\iota)_{(1)}$,
for all $i\le q$ and $r\ge 0$.
\end{theorem}

Using this result, we now derive as an application a
rather general form of the much-studied ``tangent cone inclusion".

First, we need to recall a notion originally due to Body and Sullivan,
see \cite{BMSS} and \cite{DP}.
Let $A$ be a rationally defined \cdga.  We say
$A$ has {\em positive weights}, if $A^i=\bigoplus_{j\in \Z}
A^i_j$ for each degree $i\ge 0$, and, moreover, these
vector space decompositions are compatible
with the $\cdga$ structure, and satisfy the condition
$A^1_j=0$, for all $j\le 0$.

\begin{theorem}
\label{thm:tcone}
Let $X$ be a space of finite $q$-type, and assume $\Omega^{\hdot}(X)$
has the same $q$-type as a $\cdga$ $A$ of finite $q$-type.  Then:
\begin{enumerate}
\item \label{tc1}
The tangent cone at $0$ to the resonance variety $\RR^i_r(A)$
is contained in the usual resonance variety $\RR^i_r(H^{\hdot}(A))$,
for all $i\le q$ and all $r\ge0$.
\item \label{tc2}
Suppose that, moreover, $A$ is defined over $\Q$, the
identification $H^1(\Omega^{\hdot}(X)) \cong H^1(A)$
preserves $\Q$-structures, and $A$ has positive weights.
Then $\RR^i_r(A)\subseteq \RR^i_r(H^{\hdot}(A))$, for all
$i\le q$ and all $r\ge 0$.
\end{enumerate}
\end{theorem}

\begin{proof}
Part \eqref{tc1}.
By \cite[Theorem B]{DP}, the local analytic isomorphism
from Theorem \ref{thm:b1} is induced by the exponential map
$\exp\colon H^1(X,\C)\to \TT(\pi_1(X))$.  Hence,
$\TC_0(\RR^i_r(A))$ is identified with $\TC_1(\VV^i_r(X))$.
Furthermore, since $\Omega^{\hdot}(X)\simeq_q A$,
Lemma \ref{lem:functr} allows us to identify $\RR^i_r(X)$
with $\RR^i_r(H^{\hdot}(A))$, for all $i\le q$ and all $r$.

On the other hand, by the finite
approximation result from \cite[Proposition 4.1]{PS-specres},
there is a connected, finite CW-complex $Y$ and a map
$f\colon Y\to X$ which induces an isomorphism in cohomology
up to degree $q$, and identifies the respective character tori, rank $1$ characteristic
varieties, and usual resonance varieties, again up to degree $q$.
Finally, by a result of Libgober \cite{Li}, we have that
$\TC_1(\VV^i_r(Y))\subseteq \RR^i_r(Y)$ for all $i$ and $r$.

Part \eqref{tc2}.
By \cite[Theorem C]{DP}, for each
$i\le q$, the variety $\RR^i_r(A)$ is a finite union of linear subspaces
of $H^1(A)$.  Consequently, $\TC_0(\RR^i_r(A))=\RR^i_r(A)$.
\end{proof}

\subsection{Spaces with $q$-finite models}
\label{subsec:fin mod}
As noted in \cite{DP}, there are several important classes of spaces
that satisfy the finiteness hypothesis of Theorems \ref{thm:b1} and \ref{thm:tcone}.
Let us briefly describe the main examples of spaces $X$ which are both of
finite $q$-type and admit a $\cdga$ model $A\simeq_q \Omega^{\hdot}(X)$
of finite $q$-type.

\begin{example}
\label{ex:formal}
A path-connected space $X$ is said to be {\em $q$-formal}, for some
$q\ge 1$, if $\Omega^{\hdot}(X)$ has the same $q$-type as the cohomology
ring $H^{\hdot}(X,\C)$, endowed with the $0$ differential.   Evidently, if $X$ is both
$q$-finite and $q$-formal, we may take $A=(H^{\hdot}(X,\C),d=0)$ as a
suitable model for it.

We say a discrete group $\pi$ is $q$-formal if it admits a classifying
space $K(\pi,1)$ which is $q$-formal. In group-theoretic terms,
$\pi$ is $1$-formal if and only if the Malcev--Lie algebra of $\pi$
(in the sense of Quillen \cite{Qu}) is the completion of a quadratic
Lie algebra. It is readily seen that a connected CW-complex is
$1$-formal if and only if its fundamental group is $1$-formal.
\end{example}

\begin{example}
\label{ex:qproj}
Let $X$ be an irreducible, smooth quasi-projective variety---for
short, a {\em quasi-projective manifold}. A finite model for $X$
(with $q=\infty$) can be taken to be any ``Gysin model", as constructed
by Morgan in \cite{Mo}, starting from a good compactification of $X$.
\end{example}

\begin{example}
\label{ex:nilp}
A well-known construction going back to Chevalley and 
Eilenberg \cite{CE} associates in a functorial way to each
finite-dimen\-sional Lie algebra $\f$ a ``cochain" differential 
graded algebra, $\CC^{\hdot}(\f)=  \big(\bigwedge^{\hdot} (\f^*) ,d\big)$,
with differential $d\colon \f^*\to \f^*\wedge \f^*$
equal to minus the dual of the Lie bracket, and further
extended to the exterior algebra by the graded Leibnitz rule.

Now suppose $X=K(\pi,1)$ is a classifying space for a finitely
generated, nilpotent group $\pi$.  We then may take
$A^{\hdot}=\CC^{\hdot}(\n)$, where $\n$ is the
(nilpotent) Malcev--Lie algebra associated to $\pi$,
as in \cite{Qu}.  With this choice of model, Theorem \ref{thm:b1}
holds for an arbitrary representation $\iota\colon G\to \GL(V)$
and for all $i$ and $r$, see \cite[Corollary 9.16]{DP}.
\end{example}

The next example (a particular case of the construction described 
just above) shows that the inclusions from Theorem \ref{thm:tcone} 
may well be strict.

\begin{example}
\label{ex:heis}
Let $\pi$ be the group of integral, upper diagonal $3\times 3$
matrices with $1$s on the diagonal. Then $\pi$ is a
finitely generated, torsion-free nilpotent group, with 
Malcev--Lie algebra the $3$-dimensional, $2$-step nilpotent
Lie algebra $\heis$.  In the usual coordinates, the
cochain algebra $A^{\hdot}=\CC^{\hdot}(\heis)$
may be presented as $A^{\hdot}=\bigwedge^{\hdot} (x^*,y^*,z^*)$,
with $d x^*=dy^*=0$ and $dz^*=x^*\wedge y^*$.  Clearly,
$A^{\hdot}$ is  $\infty$-finite, defined over $\Q$, and 
has positive weights ($1$ on $x^*$ and $y^*$, and $2$ on $z^*$).
On the other hand, it is readily checked that $\RR^1_1(A)=\{0\}$,
while $\RR^1_1(H^{\hdot}(A)) = \C^2$.
\end{example}

\section{Essentially rank $1$ flat connections}
\label{sect:rk1}

In this section, we single out an important subset of the parameter
space for flat connections, and study its relationship with the depth-$1$ 
resonance varieties. 

\subsection{Rank $1$ flat connections}
\label{subsec:rk1 flat}
We start with a simple definition. As before, let $A$ be a connected \cdga,
and let $\g$ be a Lie algebra.

\begin{lemma}
\label{lem:pmap}
The bilinear map $P\colon A^1\times \g\to A^1\otimes \g$,
$(\eta,g)\mapsto \eta\otimes g$ induces a map
\begin{equation}
\label{eq:pmap}
P\colon H^1(A)\times \g \to \F(A,\g).
\end{equation}
\end{lemma}

\begin{proof}
If $\eta\in Z^1(A)$,
then clearly  $\eta\otimes g$ satisfies the Maurer--Cartan
equation \eqref{eq:flat}.  Thus, $P$ restricts
to a map $P\colon Z^1(A)\times \g \to \F(A,\g)$.  Now,
since $A$ is connected, $Z^1(A) =H^1(A)$, and so we are done.
\end{proof}

\begin{definition}
\label{def:flat1}
The {\em essentially rank one}\/ part of $\F(A,\g)$ is the set
$\F^1(A,\g):=P(H^1(A)\times \g)$. We call its complement
the {\em regular}\/ part of $\F(A,\g)$.
\end{definition}

Now suppose $\theta\colon \g\to \gl (V)$ is a finite-dimensional
representation.  We may then single out a subset of $\F^1(A,\g)$, 
defined as
\begin{equation}
\label{eq:piatheta}
\Pi(A,\theta)=P(H^1(A)\times V(\det\circ \theta)),
\end{equation}
where $\det\colon \gl(V)\to \C$ is the determinant, and
$V(\det\circ \theta)=\{ g\in \g \mid \det(\theta(g))=0\}$.

\begin{lemma}
\label{lem:rk1 zar}
Suppose $A$ is $1$-finite and $\g$ is finite-dimensional. Then,
\begin{enumerate}
\item \label{r1}
$\F^1(A,\g)$ is an irreducible, Zariski-closed subset of $\F(A,\g)$ 
containing $0$.
\item \label{r2}
$\Pi(A,\theta)$ is a Zariski-closed subset of $\F^1(A,\g)$ 
containing $0$.
\end{enumerate}
\end{lemma}

\begin{proof}
Plainly, $\F^1(A,\g)$ is either $\{0\}$, or the cone on
$\PP(H^1(A)) \times \PP(\g)$.  The other claims follow
at once.
\end{proof}

The above definitions allow us to describe the $0$-th resonance
variety of a $\cdga$ in a simple, yet important situation.
\begin{lemma}
\label{lem:res0}
Suppose $A$ is $1$-finite, $\g=\sl_2$, and $V=\C^2$ is the defining
representation, given by the inclusion $\theta\colon \sl_2\inj \gl_2$.
Then $\RR^0_1(A,\theta)=\Pi(A,\theta)$.  Moreover, the variety 
$\RR^0_1(A,\theta)$  is irreducible, and its dimension is positive, 
provided $H^1(A)\ne 0$.
\end{lemma}

\begin{proof}
Pick a basis $\{\eta_i\}$ for $A^1$, and let
$\omega =\sum_i \eta_i \otimes g_i $ be a
flat $\sl_2$-connection.  From Lemma \ref{lem:aom0}, we know that $\omega$
belongs to $\RR^0_1(A,\theta)$ if and only if $\bigcap_i \ker(\theta(g_i))$
contains a non-zero vector in $\C^2$. In coordinates, this means
that all matrices $g_i$ are of the form $\left( \begin{smallmatrix}
0 & \lambda_i \\ 0&0\end{smallmatrix}\right)$, with
$\sum_i \lambda_i \eta_i \in H^1(A)$. It follows that
$\RR^0_1(A,\theta)=\Pi(A, \theta)$,
as claimed.  The remaining assertions follow from the fact
that $V(\det\circ \theta)$ is an irreducible hypersurface
in $\sl_2$.
\end{proof}

\subsection{Rank $1$ flat connections and resonance}
\label{subsec:rk1 res}

The next results describe in full generality the way in which
the essentially rank $1$ part of the set of flat connections,
$\F^1(A,\g)$, cuts the depth $1$ resonance varieties
$\RR^k_1(A,\theta)$. In view of Remark \ref{rem:univ},
we start with the ``universal" case.

\begin{theorem}
\label{thm:ess1res}
Let $A$ be a connected $\cdga$.
Suppose $\g=\gl_n$ and $\theta= \id_{\g}$, and
let $\omega=\eta\otimes b$ be an arbitrary element of
$\F^1(A,\g)$.  Then $\omega$ belongs to
$\RR^k_1(A,\theta)$ if and only if there is an
eigenvalue $\lambda$ of $b$ such that $\lambda \eta$
belongs to $\RR^k_1(A)$.
\end{theorem}

\begin{proof}
We start by putting the matrix $b$ in Jordan canonical form,
so that $\C^n=\bigoplus_i V_i$, where $V_i$ are the generalized
eigenspaces of $b$, and $b=\bigoplus_i b_i$, where $b_i$ are
the Jordan blocks.  Setting $\omega_i=\eta\otimes b_i$, the
Aomoto complex splits accordingly, and so does its cohomology,
\begin{equation}
\label{eq:hka}
H^k(A\otimes \C^n, d_{\omega}) = \bigoplus_i H^k(A\otimes V_i, d_{\omega_i}).
\end{equation}
Thus, it is enough to consider the case when $b$ is a single
Jordan block, of size $n\ge 2$, and with eigenvalue $\lambda$.

In the standard basis $e_1,\dots, e_n$ for $\C^n$, the
differential $d_{\omega}$ is given by
\begin{equation}
\label{eq:dw}
d_{\omega} \Big(\sum_{j=1}^n \alpha_j\otimes e_j \Big)=
\sum_{j<n}  (d\alpha_j + \lambda \eta \alpha_j + \eta\alpha_{j+1}) \otimes e_j
+ ( d\alpha_n + \lambda \eta \alpha_n ) \otimes e_n
\end{equation}

Let $\C^{n-1}$ be the subspace of $\C^n$ spanned by
$e_1,\dots, e_{n-1}$. Clearly, $A\otimes \C^{n-1}$ is a
subcomplex of $A\otimes \C^{n}$ with respect to
$d_{\omega}$. Moreover, the induced differential
on the quotient complex $A\otimes \C \cdot e_n$ may be
identified with the rank $1$ differential
$d_{\lambda \eta} \colon A\to A$ sending $\alpha\in A$
to $d\alpha + \lambda \eta \alpha$.

Using \eqref{eq:dw}, it is readily verified that the
connecting homomorphism in the associated
long exact sequence in cohomology,
$\delta_{\hdot} \colon H^{\hdot}(A\otimes \C, d_{\lambda \eta})
\to H^{\hdot\;+1}(A\otimes \C^{n-1}, d_{\omega})$, sends the class of
$\alpha\in A$ to the class of $\eta \alpha \otimes  e_{n-1}$.
It remains to establish the following equivalence
\begin{equation}
\label{eq:equiv}
H^k(A\otimes \C^n, d_{\omega}) =0 \same
H^k(A\otimes \C, d_{\lambda\eta}) =0.
\end{equation}

To prove the backwards implication, we may iterate
the above construction and use the associated long exact
sequences to infer that the natural map,
$H^k(A\otimes \C\cdot e_1, d_{\omega}) \to 
H^k(A\otimes \C^n, d_{\omega})$, 
is onto. On the other hand, it follows from \eqref{eq:dw} that
$(A\otimes \C\cdot e_1, d_{\omega})$ can be naturally identified
with $(A, d_{\lambda\eta})$, and the claim follows. 

To prove the forward implication, we first check
that, for any $d_{\lambda\eta}$-cocycle $\alpha\in A^k$,
there is a $\beta\in A^k$ such that $\eta\alpha + d_{\lambda\eta}(\beta)=0$.
Indeed, the class of the $d_{\omega}$-cocycle $\alpha\otimes e_1$
in $H^k(A\otimes \C^{n-1}, d_{\omega})$ lies in the image of
$\delta_{k-1}$, by our vanishing assumption. This means that
there is a $d_{\lambda\eta}$-cocycle $\beta'\in A^{k-1}$ such that
$\alpha\otimes e_1 = \eta \beta' \otimes e_{n-1} +
\sum_{j<n} d_{\omega} (\beta'_j \otimes e_j)$,
for some $\beta'_j\in A^{k-1}$.  By \eqref{eq:dw} and
graded-commutativity, it follows that $\eta \alpha\otimes e_1 =
\sum_{j<n} \eta d\beta'_j \otimes e_j$. Hence,
\begin{equation*}
\label{eq:ab1}
\eta \alpha=
\eta d\beta'_1=-d(\eta\beta'_1)=d_{\lambda\eta}(-\eta\beta'_1),
\end{equation*}
where the second equality comes from the (graded) Leibnitz rule.
Thus, we may take $\beta=\eta\beta'_1$.

Now consider the element $\beta\otimes e_1 +\alpha\otimes e_2\in
A^k \otimes \C^n$, where $\alpha$ is an arbitrary $d_{\lambda\eta}$-cocycle
in $A^k$, and $\beta$ is chosen so that $\eta\alpha + d_{\lambda\eta}(\beta)=0$.
Using again \eqref{eq:dw}, it is readily verified that
$d_{\omega}(\beta\otimes e_1 +\alpha\otimes e_2)=0$,
from which we infer that $\beta\otimes e_1 +\alpha\otimes e_2 =
d_{\omega} (\sum_{j=1}^{n} \beta_j \otimes e_j)$. Therefore,
$\alpha = d_{\lambda\eta}(\beta_2)+\eta\beta_3$, where
$\beta_3=0$ if $n=2$.  It follows that $\delta_k$ sends the
class of $\alpha$ in $H^k(A\otimes \C, d_{\lambda\eta})$ to $0$.
Hence, $\delta_k$ is the zero map, which implies that
$H^k(A\otimes \C, d_{\lambda\eta})=0$, since
$H^k(A\otimes \C^n, d_{\omega})=0$.
This completes the proof.
\end{proof}

\begin{remark}
\label{rem:spec seq}
As suggested by the referee, the above theorem can be interpreted 
in terms of the spectral sequence $(E^k,d^k)$ associated to the differential 
module $C= (A\otimes \C^n,d_{\omega})$, endowed with the finite, increasing 
filtration defined by $F_0=0$ and $F_s= A\otimes \spn \{e_1,\dots ,e_s\}$ 
for $1\le s\le n$ and differential $d_{\omega}$ of degree $1$.   
The degree $q$ part of the first page of this (convergent) spectral sequence has the form 
$\bigoplus_{\lambda}  \bigoplus_{n_\lambda}  H^q(A, d_{\lambda \eta})$, 
where $\lambda$ runs through the eigenvalues of the matrix $b$, and $n_{\lambda}$ 
is the multiplicity of $\lambda$.  Theorem \ref{thm:ess1res} can then be 
restated as saying:  
\begin{equation}
\label{eq:ss vanish}
H^q(C)=0 \same \text{ the degree $q$ part of $E^1$ vanishes. }
\end{equation}

The backwards implication is obvious, since 
$\gr_s(H^q(C))$ is isomorphic to the degree $q$ part of $E^{\infty}_s$, 
but the forward implication 
is quite subtle, and cannot be proved by a simple degeneration
argument.  Indeed, as the next example shows, the presence 
of non-trivial higher Massey products in the $\cdga$ $A$  may very well
yield non-trivial differentials in the $E^2$ page.  
\end{remark}

\begin{example}
\label{ex:heis ss}
Let $A=\big(\bigwedge (x,y,z); dx=dy=0, dz=xy\big)$ be the minimal model 
of the Heisen\-berg $3$-dimensional nilmanifold, and consider the
spectral sequence associated to the essentially rank $1$ flat connection 
$\omega = y \otimes b$, where   
$b=\left( \begin{smallmatrix} 0 & 1 & 0 \\ 
0 & 0 & 1 \\ 0 &  0 & 0 \end{smallmatrix}\right)$. 
We claim that the differential $d^2\colon E^2_3\to E^2_1$ is non-zero.

To prove the claim, consider the (degree $1$) class of  
$\alpha = z\otimes e_2 + x\otimes e_3 \in Z^2_3$ in $E^2_3$, 
and suppose that $d^2 ([\alpha]) =0\in E^2_1$. 
Our assumption implies that 
$d_{\omega} (\alpha) \in Z^1_0 + d_{\omega} Z^1_2= d_{\omega} Z^1_2$.
Hence, there is an element $\beta =\beta_1\otimes e_1 + \beta_2\otimes e_2$
with $d \beta_2=0$ such that $d_{\omega} (\alpha) = d_{\omega} (\beta)$. 
Equivalently, we have $yz = d\beta_1 +y \beta_2$, for
some $\beta_1 \in A^1$ and $\beta_2 \in A^1$ with $d \beta_2=0$.
It follows that $yz = d(c_1 z)+ y(c_2 x+ c_3y)$, for some constants $c_i\in \C$.
This leads to the equality $yz = (c_1-c_2)xy$ in $A^2$, a contradiction.
\end{example}

We now turn to the general case, in which $A$ is a connected $\cdga$,
$\g$ is an arbitrary Lie algebra, and $\theta\colon \g\to \gl (V)$ is a 
finite-dimensional representation.

\begin{corollary}
\label{cor:vdet}
Let $\omega=\eta\otimes g$ be an arbitrary element of
$\F^1(A, \g)$.  Then $\omega$ belongs to $\RR^k_1(A,\theta)$
if and only if there is an eigenvalue $\lambda$ of $\theta(g)$ such that
$\lambda \eta$ belongs to $\RR^k_1(A)$. Moreover,
\begin{equation}
\label{eq:pi res}
\Pi(A,\theta)\subseteq \bigcap_{k: H^k(A)\ne 0} \RR^k_1(A, \theta).
\end{equation}
\end{corollary}

\begin{proof}
In view of Remark \ref{rem:univ}, the element $\omega$ belongs to
$\RR^k_1(A,\theta)$ if and only if $\eta \otimes \theta(g)$ belongs
to $\RR^k_1(A,\id_{\gl(V)})$. Hence, the first claim follows from
Theorem \ref{thm:ess1res}.

To prove the second claim, start by noting that $\theta(g)$ has the
eigenvalue $\lambda=0$. By \eqref{eq:basept},  then,
$\lambda \eta\in \RR^k_1(A)$ if and only if $H^k(A)\ne 0$,
and this completes the proof.
\end{proof}

\subsection{Some representation theory}
\label{subset:rep theory}
Next, we analyze the sets $V(\det \circ \theta)$ in the case when
$\theta\colon \g\to \gl(V)$ is the structure map of a (non-zero)
finite-dimensional module over $\g=\sl_2$.  We refer the
reader to \cite{FH, Hu} for the necessary ingredients
from classical representation theory.

Let $\{H, X, Y\}$ be the standard basis of $\sl_2$,
and let $\cart=\C\cdot H$ be the standard Cartan subalgebra.
The representation $V$ decomposes as
a direct sum of irreducible representations of the form $V(n)$,
with $n\ge 0$, where $\dim V(n)=n+1$, and the eigenvalues of
$H$ on $V(n)$ are $n, n-2,\dots , -n$. We denote the structure
map of $V(n)$ by $\theta_n$.  The defining representation of $\sl_2$
is $V(1)$, while $\theta_1$ is the inclusion $\sl_2\hookrightarrow \gl_2$.
Finally, we denote the map $\det \circ\theta_1\colon \sl_2\to \C$
simply by $\det$.

\begin{lemma}
\label{lem:det}
If $V$ has a direct summand equal to $V(n)$, with $n$ even,
then $V(\det\circ\theta)=\sl_2$. Otherwise,
$V(\det\circ\theta)=V(\det)$.
\end{lemma}

\begin{proof}
It is enough to check that $\det\circ\theta_n$ is equal to $0$
if $n$ is even, and is a non-zero multiple of $\det^{(n+1)/2}$
if $n$ is odd.  In order to verify this assertion, we shall use some
basic invariant theory.

Let $\iota\colon \SL_2 \to \GL(V)$ be the (unique) rational representation
for which $d_1(\iota) =\theta$.   We first claim that
$\det\circ\theta$ belongs to $\C[\sl_2^*]^{\SL_2}$,
the invariant subalgebra of the polynomial algebra on the
dual vector space to $\sl_2$, taken with respect to the
adjoint representation of $\SL_2$.  Indeed, note that
$\iota \circ c(s) = c(\iota (s)) \circ \iota$, for any $s\in \SL_2$,
where $c(-)$ stands for the conjugation action.
The claim then follows upon taking differential at $1$.

Next, let $W=\Z_2$ be the Weyl group of $\SL_2$, acting on
$\cart$ by the alternating representation, and consider
the natural morphism $r\colon \C[\sl_2^*]^{\SL_2} \to \C[\cart^*]^W$
between the respective subalgebras of invariants. Clearly,
$r$ is surjective, since
$\C[\cart^*]^W=\C[(H^{*})^2]$ and $r(\det)= -(H^{*})^2$.
Furthermore, $r$ is also injective, since every element of $\sl_2$
outside of $V(\det)$ is $\SL_2$-conjugate to an element of $\cart$.
Thus, $\C[\sl_2^*]^{\SL_2} = \C[\det]$.

Finally, assume $V=V(n)$. Since $\deg (\det\circ \theta_n) =n+1$,
we must have $\det\circ \theta_n=0$ if $n$ is even, and
$\det\circ \theta_n=\lambda_n \det^{(n+1)/2}$ if $n$ is odd.
But $\det(\theta_n(H))\ne 0$, and so $\lambda_n\ne 0$, thus
completing the proof.
\end{proof}

\section{Holonomy Lie algebras of differential graded algebras}
\label{sect:holo lie}

Our primary aim in this section is to extend the definition
of holonomy Lie algebras to arbitrary \cdga's, and to
extend several results from \cite{DPS}, thereby relating
flat connections to Lie algebra representations and
resonance varieties of \cdga's to resonance varieties
of Lie algebras. As a first application, we compute the
depth-$1$ resonance varieties of finite-dimensional,
nilpotent Lie algebras, in two important non-abelian
cases.

\subsection{Holonomy Lie algebra}
\label{subsec:holo}

Let $A^{\hdot}$ be a connected \cdga.  For our purposes here,
we will be mainly interested in the differential $d\colon A^1\to A^2$
and the product $\cup\colon A^1\wedge A^1\to A^2$. Our
starting point is the following lemma, whose proof is
straightforward.

\begin{lemma}
\label{lem:cores}
For any representation $\theta\colon \g \to \gl(V)$, both the
parameter space $\F(A,\g)$ and the resonance varieties
$\RR^0_r(A,\theta)$ and $\RR^1_r(A,\theta)$ depend only
on the co-restriction of $d$ and $\cup$ to the subspace
$\im(d)+\im (\cup)\subseteq A^2$.
\end{lemma}

This leads us to the following construction. Given a $\C$-vector
space $W$, denote by $\L^{\hdot}(W)$ the free Lie algebra on $W$,
graded by bracket length. In low degrees, $\L^{1}(W)=W$, 
while $\L^{2}(W)$ may be identified with $W\wedge W$ via
$[u,v] \leftrightarrow u\wedge v$.

Now let $A$ be a $1$-finite \cdga.  Set $A_i=(A^i)^*$, and let
$\L(A_1)$ be the free Lie algebra on the dual vector space $A_1$.
We then have dual maps, $d^*\colon A_2\to A_1=\L^1(A_1)$
and $\cup^*\colon A_2\to A_1\wedge A_1=\L^2(A_1)$.

\begin{definition}
\label{def:hold}
The {\em holonomy Lie algebra}\/ of a $1$-finite $\cdga$ $A$
is the finitely presented Lie algebra
$\h(A) = \L(A_1) / \ideal(\im \partial_A)$,
where $\partial_A$ is the linear map
\begin{equation}
\label{eq:partial}
\partial_A:=d^*+\cup^* \colon A_2 \to
\L^1(A_1) \oplus \L^2(A_1)\subset \L(A_1).
\end{equation}
\end{definition}

This construction is functorial.  Indeed, if $\varphi\colon A\to A'$
is a $\cdga$ map, then the linear map $\varphi_1=(\varphi^1)^*\colon A'_1\to A_1$
extends to Lie algebra map $\L(\varphi_1)\colon \L(A'_1)\to \L(A_1)$,
which in turn induces a Lie algebra map $\h(\varphi)\colon \h(A')\to \h(A)$.

\begin{remark}
\label{rem:holo chen}
In the case when $d=0$, the above definition recovers the classical
definition of the holonomy Lie algebra of an algebra, due to
K.~T.~Chen \cite{Ch}.  In this situation, $\h(A)$ inherits a natural grading
from the free Lie algebra, compatible with the Lie bracket; thus,
$\h(A)$ is a finitely presented graded Lie algebra,
with generators in degree $1$, and relations in degree $2$.
In general, though, the ideal generated by $\im(\partial_A)$
is not homogeneous, and the Lie algebra $\h(A)$ is not graded.
\end{remark}

\begin{example}
\label{ex:holc}
Let $\f$ be a finite-dimensional Lie algebra, and let $A^{\hdot}=
\CC^{\hdot}(\f)$ be its cochain \cdga, as in Example \ref{ex:nilp}.
Then $\h(A)$ is the quotient of $\L(\f)$ by all the relations of
the form $[x,y]_{\f} = [x,y]_{\L(\f)}$, with $x,y\in \f$. Hence,
$\h(\CC^{\hdot}(\f))=\f$, as (ungraded) Lie algebras.
\end{example}

\subsection{Flat connections and representations}
\label{subset:flatrep}
We now relate the set of flat connections
on a $\cdga$ with the set of Lie algebra representations
of the corresponding holonomy Lie algebra.

To this end, let us start with an arbitrary finitely-generated Lie
algebra $\h$, presented as the quotient $\L(A_1)/\ideal(\im \partial)$
of the free Lie algebra on a finite-dimensional vector space $A_1$
by the ideal generated by the image of a linear map,
$\partial\colon A_2\to \L(A_1)$.

Now let $\g$ be a  Lie algebra, and let
$\Hom_{\Lie} (\h, \g)$ be the set of Lie algebra morphisms from
$\h$ to $\g$.  Every such morphism defines by
restriction a linear map from $A_1$ to $\g$. Thus, we have
a canonical inclusion,
\begin{equation}
\label{eq:rep lie}
\Hom_{\Lie} (\h, \g)\subseteq \Hom(A_1,\g).
\end{equation}

We let $\Hom_{\Lie}^1 (\h, \g)$ denote the intersection of 
$\Hom_{\Lie} (\h, \g)$ with $\Hom^1 (A_1, \g)$.
When $\g$ is finite-dimensional, it is readily checked that 
$\Hom_{\Lie} (\h, \g)$ is a Zariski-closed
subset of the affine space $\Hom(A_1,\g)$.  This endows the
{\em representation variety}\/ $\Hom_{\Lie} (\h, \g)$ with an affine
structure, natural in both $\h$ and $\g$. 

The next result extends Lemma 3.13 from \cite{DPS}, where the
case in which $A$ has zero differential was analyzed.

\begin{prop}
\label{prop:flathol}
Let $A$ be a $1$-finite $\cdga$, and let $\g$ be a Lie algebra.
Then, the canonical isomorphism
$\psi\colon A^1\otimes \g \isom  \Hom (A_1,\g)$
restricts to an identification
\begin{equation}
\label{eq:psi1}
\xymatrix{\psi\colon \F(A,\g) \ar^(.43){\simeq}[r]&  \Hom_{\Lie} (\h(A), \g).}
\end{equation}
Moreover,  $\psi$ further restricts to an identification
\begin{equation}
\label{eq:psi2}
\xymatrix{\psi\colon \F^1(A,\g) \ar^(.43){\simeq}[r]&  \Hom^1_{\Lie} (\h(A), \g).}
\end{equation}
\end{prop}

\begin{proof}
To prove the first claim, we need to show the following: an
element $\omega\in A^1 \otimes \g$ satisfies the Maurer--Cartan
equation \eqref{eq:flat} if and only if the extension of the map
$\rho=\psi(\omega)$ to the free Lie algebra, $\rho\colon \L(A_1) \to \g$,
vanishes on $\partial_A(a)$, for all $a\in A_2$.

To that end, pick a basis $\{\eta_i\}$ for $A^1$ and denote by $\{\eta^*_i\}$
the dual basis for $A_1$. Writing $\omega=\sum_i \eta_i \otimes g_i$,
we have that $\rho(\eta^*_i)=g_i$.  In the chosen bases, the linear
maps $d^*\colon A_2\to A_1$ and $\cup^*\colon A_2\to A_1\wedge A_1$
may be written as $d^*(a)= \sum_i d_i \eta^*_i$ and
$\cup^*(a)= \sum_{i<j} \mu_{ij} \eta^*_i \wedge \eta^*_j$.
Recalling that $\partial_A=d^*+\cup^*$, we find that
\begin{equation}
\label{eq:lierels}
\rho \partial_A  (a) =  \sum_i d_i g_i  + \sum_{i<j} \mu_{ij} [g_i,g_j].
\end{equation}

In the same bases, the Maurer--Cartan equation for $\omega$ can be
written as \eqref{eq:flat coords}.  Applying the (injective) linear map $\psi$
to both sides of this equation, we find that \eqref{eq:flat coords} is
equivalent to
\begin{equation}
\label{eq:mc}
\sum_i \langle d\eta_i,  a\rangle g_i  +
\sum_{i<j} \langle \cup(\eta_i \wedge \eta_j),  a\rangle   [g_i,g_j]=0,
\end{equation}
for all $a\in A_2$.  Clearly, the right side of \eqref{eq:lierels}
equals the left side of \eqref{eq:mc}, and so the first claim is proved.

As for the second claim, we need to show that $\omega$ belongs to
$\F^1(A,\g)$, that is, $\omega$ can be written as $\eta\otimes g$, 
for some $\eta\in A^1$ and $g\in \g$, if and only if $\rho=\psi(\omega)$ 
belongs to $\Hom^1(A_1,\g)$, that is, $\im(\rho)$ has dimension at 
most $1$.  But this is clear, and we are done.
\end{proof}

\subsection{Lie algebra cohomology}
\label{subsec:lie coho}
Next, we relate the cohomology in low degrees of the holonomy
Lie algebra to the cohomology of the corresponding Aomoto complex.
For basic facts about cohomology of Lie algebras, we refer
to the book by Hilton and Stammbach \cite{HS}.

We start with an arbitrary Lie algebra $\h$, presented as
$\h=\L(A_1)/\ideal(\im \partial\colon A_2\to \L(A_1))$, and
with an arbitrary $\h$-module $V=V_{\rho}$, with structure
map $\rho\colon \h\to \gl(V)$.  Let $\Der(\h, V)$
be the vector space consisting of all linear maps
$\delta\colon \h\to V$ satisfying
$\delta([g,h])=\rho(g)\cdot \delta(h)-\rho(h)\cdot \delta(g)$.
Let $\pi\colon \L(A_1)\to \h$ be the canonical projection;
restricting to the free generators identifies
$\Der(\L(A_1), V_{\rho\pi})$ with $\Hom(A_1,V)$.

Let $d^1_{\rho}\colon \Hom(A_1,V)\to  \Hom(A_2,V)$ be the
linear map sending $\delta$ to $\delta\circ \partial$;
the kernel of this map can be identified with the space
$\Der(\h,V)$. Set $A_0=\C$, and denote by
$d^0_{\rho}\colon \Hom(A_0,V)\to  \Hom(A_1,V)$ the
linear map sending $v\in V$ to the inner derivation
$\delta_v$, defined by $\delta_v(h)=\rho h(v)$, for $h\in \h$.
Then, the cochain complex
\begin{equation}
\label{eq:hlie}
\xymatrixcolsep{22pt}
\xymatrix{0\ar[r]& \Hom(A_0,V) \ar^{d^0_{\rho}}[r]
& \Hom(A_1,V) \ar^{d^1_{\rho}}[r]
& \Hom(A_2,V) \ar[r]& 0 \ar[r]& \cdots}
\end{equation}
computes the cohomology groups $H^i(\h,V)$, for $i=0$ and $1$.

Now let $A$ be a $1$-finite \cdga, and let $\h=\h(A)$ be its
holonomy Lie algebra. Consider the Lie algebra $\g=\gl(V)$,
endowed with the representation $\theta=\id_{\gl(V)}$,
and let
$\psi\colon \F(A,\gl(V)) \isom  \Hom_{\Lie} (\h(A), \gl(V))$ be
the isomorphism provided by Proposition \ref{prop:flathol}.

\begin{lemma}
\label{lem:coho lie}
With notation as above, let $\omega \in \F(A,\gl(V))$ be a flat
connection, and let $\rho=\psi(\omega)$ be the corresponding
structure map for the $\h(A)$-module $V=V_{\rho}$.
There is then an isomorphism
$ H^i(A\otimes V, d_{\omega}) \cong H^i(\h(A), V_{\rho})$
for $i=0$ and $1$.
\end{lemma}

\begin{proof}
Consider the Aomoto complex \eqref{eq:aomoto} and
the cochain complex \eqref{eq:hlie}.  It is readily seen
that, in the range $i\le 1$, the natural monomorphism
$\psi \colon A^{i}\otimes V \to \Hom(A_{i},V)$
is an isomorphism, and commutes with the differentials
$d^i_{\omega}$ and $d^i_{\rho}$.  The conclusion
follows at once.
\end{proof}

\subsection{Resonance varieties of a Lie algebra}
\label{subsec:res lie}
We now consider the Lie analogue of the resonance
varieties, following the approach from \cite{DPS}.
Let $\h$ be a Lie algebra, and let $\theta\colon \g\to \gl(V)$
be a representation of another Lie algebra.  Associated to
these data we have the resonance varieties
\begin{equation}
\label{eq:lie res}
\RR^i_r(\h, \theta)= \{\rho \in \Hom_{\Lie}(\h,\g)
\mid \dim_{\C} H^i(\h,  V_{\theta\rho}) \ge  r\}.
\end{equation}

\begin{lemma}
\label{lem:resclosed}
Suppose $\h$ is a finitely generated Lie algebra
and $\theta\colon \g \to \gl(V)$ is a finite-dimen\-sional
representation of a finite-dimensional Lie algebra.
Then the resonance varieties $\RR^i_r(\h,\theta)$
are Zariski-closed subsets of $\Hom_{\Lie}(\h,\g)$,
for $i\le 1$ and $r\ge 0$.
\end{lemma}

\begin{proof}
Let $\rho\in \Hom_{\Lie}(\h,\g)$, and consider the cochain
complex \eqref{eq:hlie}, with $\rho$ replaced by $\theta\circ \rho$.
It is easy to check that the differentials $d^i_{\theta\rho}$ depend
algebraically on $\rho$. The claim then follows from \cite[Lemma 9.2]{DP}.
\end{proof}

The next result generalizes Corollary 3.18 from \cite{DPS},
where only \cdga's with differential $d=0$ were considered.

\begin{corollary}
\label{cor:lie res}
Given a $1$-finite $\cdga$ $A$ and a Lie algebra
representation $\theta\colon \g \to \gl(V)$, the canonical
isomorphism $\psi\colon \F(A,\g) \isom  \Hom_{\Lie} (\h(A), \g)$
restricts to an identification
\begin{equation*}
\label{eq:lie res iso}
\RR^i_r(A,\theta) \isom  \RR^i_r (\h(A), \theta),
\end{equation*}
for each $i\le 1$ and $r\ge 0$.
\end{corollary}

\begin{proof}
Let $\omega\in \F(A,\g)$ be a flat connection,
and let $\rho=\psi(\omega)\colon  \h(A)\to \g$
be the corresponding  Lie algebra morphism.
Then $\omega'=(\id_{A^1}\otimes \theta) (\omega)\in \F(A,\gl(V))$
corresponds to $\theta\circ \rho \colon  \h(A)\to \gl(V)$, and so
Lemma \ref{lem:coho lie} provides us with an isomorphism
$ H^i(A\otimes V, d_{\omega'}) \cong H^i(\h(A), V_{\theta \rho})$,
for $i=0$ and $1$.  The desired conclusion now follows from
Remark \ref{rem:univ}.
\end{proof}

\subsection{Finite-dimensional Lie algebras}
\label{subsec:fg lie}
We now consider in more detail the case when $\f$ is a Lie algebra
of finite dimension. Let $V_{\rho}$ be an arbitrary $\f$-module.
Using the standard $\f$-resolution of $\C$, we see that
$H^\hdot(\f, V_\rho)=H^\hdot (\CC (\f)\otimes V, d_\rho)$.

\begin{lemma}
\label{lem:cohlie}
Given a finite-dimensional Lie algebra $\f$ and a representation 
$\theta\colon \g \to \gl(V)$, let $\omega \in \F(\CC(\f), \g)$ correspond to 
$\rho \in \Hom_{\Lie}(\f, \g)$ under the isomorphism $\psi$ from 
Proposition \ref{prop:flathol}. Then $d_{\theta \circ \rho} =d_{\omega}$.
\end{lemma}

\begin{proof}
Straightforward direct computation.
\end{proof}

\begin{corollary}
\label{cor:liecoh}
For a finite-dimensional Lie algebra $\f$ and a representation $\theta\colon \g\to \gl(V)$,
the canonical isomorphism $\psi\colon \F(\CC(\f), \g) \isom \Hom_{\Lie}(\f, \g)$
identifies $\RR^i_r (\CC(\f), \theta)$ with $\RR^i_r (\f, \theta)$, for all $i, r \geq 0$.
\end{corollary}

Note also that, if both $V$ and $\g$ are finite-dimensional, then $\RR^i_r(\f, \theta)$
is a Zariski-closed subset of $\Hom_{\Lie}(\f, \g)$, for all $i, r \geq 0$.

\subsection{Lie algebras and rank one resonance}
\label{ss45c}

We now revisit the ubiquitous Lie algebra $\sl_2$. We denote by $\sol_2$ its
standard Borel subalgebra, generated by $H$ and $X$. This $2$-dimensional
solvable Lie algebra is the unique nonabelian Lie algebra in dimension $2$.
Setting $h=\frac{1}{2}H$ and $x=X$, we find that
$\CC^\hdot (\sol_2)=\bigwedge^\hdot (h^*, x^*)$ with $dh^*=0$ and
$dx^*=x^* \wedge h^*$.

\begin{remark}
\label{rem:comm}
Later on, we will need the following easily verified fact.
Suppose $\g=\sl_2$ or $\sol_2$, and $g, g' \in \g$;
then $[g, g']=0$ if and only if $\rank \{g, g' \} \leq 1$.
\end{remark}

As shown in the next proposition, non-triviality of rank $1$ resonance is
detected by the Lie algebras $\sol_2$ and $\heis$.

\begin{prop}
\label{prop:detect}
Let $A^\hdot$ be a $1$-finite \cdga. Then:
\begin{enumerate}
\item \label{ct1}
$\RR^1_1(A) \nsubseteq \{0\}$ if and only if
$\Hom_{\Lie}(\h(A), \sol_2)$ contains a surjection.
\item \label{ct2}
$\RR^1_1(H^\hdot (A)) \nsubseteq \{0\}$ if and only if
$\Hom_{\Lie}(\h(A), \heis)$ contains a surjection.
\end{enumerate}
\end{prop}

\begin{proof}
By definition of resonance, the condition that $\RR^1_1(A) \nsubseteq \{0\}$
is equivalent to the existence of elements $\alpha, \beta \in A^1$ with
$\rank \{\alpha, \beta\}=2$, $d \alpha=0$, and $d \beta+\alpha \beta=0$.
By the above description of $\CC^\hdot(\sol_2)$, this is equivalent
to the existence of a \cdga~morphism $\phi \colon \CC^\hdot(\sol_2)\to  A^\hdot$
such  that $\phi^1$ is injective.

Likewise, the condition that $\RR^1_1(H^\hdot (A)) \nsubseteq \{0\}$
means that there exist $\alpha, \beta, \gamma  \in A^1$ such that
$\rank \{\alpha, \beta\}=2$, $d \alpha=d \beta=0$, and $\alpha \beta =d \gamma$,
or, equivalently, there exists a \cdga~morphism  $\phi \colon \CC^\hdot(\heis)\to A^\hdot$
such that $H^1(\phi)$ is injective.

Now let $\f$ be a finite-dimensional Lie algebra. View
$\Hom_{\cdga}(\CC^\hdot (\f), A^\hdot)$ inside
$\Hom(\f^*, A^1)$ and $\Hom_{\Lie}(\h(A), \f)$ inside
$\Hom(A_1, \f)$, and consider the natural identification
$\Hom(\f^*, A^1) \overset{\sim} \longleftrightarrow \Hom(A_1, \f)$,
given by taking the transpose of matrices.  Lemma 3.4 from \cite{DP}
and Proposition \ref{prop:flathol} together imply that
$\Hom_{\cdga}(\CC^\hdot (\f), A^\hdot)$ is identified in this way with
$\Hom_{\Lie}(\h(A), \f)$. In both cases under consideration,
denote by $\rho$ the morphism of Lie algebras corresponding to $\phi$.

In the first case, the injectivity condition translates into the surjectivity
of the restriction of $\rho \colon \h(A) \rightarrow \sol_2$ to the 
generating set $A_1$ for the holonomy Lie algebra $\h(A)$. 
Since $\sol_2$ has dimension $2$, this last condition is equivalent 
to the surjectivity of $\rho$, as claimed.

In the second case, let $\ab\colon \heis\to \C^2$ be the abelianization map.  
We may then rephrase the injectivity of $H^1 (\phi)$ as
the surjectivity of the restriction of $\ab \circ \rho$ to $A_1$. 
Since the $3$-dimensional Heisenberg Lie algebra
$\heis$ is nilpotent, and since a Lie morphism between nilpotent Lie algebras
is onto if and only if its abelianization is onto, this last condition is again
equivalent to the surjectivity of $\rho$, as claimed.
\end{proof}

\begin{remark}
\label{rem:res stuff}
In general, there is no implication between the vanishing resonance 
properties \eqref{ct1} and \eqref{ct2} from Proposition \ref{prop:detect}.  
For instance, if $A^\hdot = \CC ^\hdot (\sol_2)$, then $H^1(A)=\C$ 
and $\RR_1^1(A)=\{0,1\}$, while $\RR_1^1(H^\hdot (A))=\{0\}$.
On the other hand, if $A^\hdot = \CC ^\hdot (\heis)$, then $H^1(A)=\C^2$
and $\RR_1^1(A)=\{0\}$, while $\RR_1^1(H^\hdot(A))=\C^2$.
Nevertheless, it follows from Theorem \ref{thm:tcone}(\ref{tc2}) that,
under certain additional assumptions, one such implication holds.
\end{remark}

\subsection{Nilpotent Lie algebras}
\label{subsec:nilp}
To conclude this section, we compute the germs of the depth $1$
characteristic varieties of a finitely-generated nilpotent group $\pi$,
relative to an arbitrary rational representation of a Lie group $G$, where
$G$ is either semisimple of rank $1$, or one of the Borel subgroups 
of such a group. As explained in Example \ref{ex:nilp}, such
a computation follows from the global computation
of depth $1$ resonance varieties of the cochain algebra of the
corresponding finite-dimensional, nilpotent Lie algebra $\n$,
relative to a representation of a Lie algebra $\g$, where
$\g$ is either $\sl_2$ or $\sol_2$.

\begin{lemma}
\label{lem:nilpflat}
Let $\n$ be a finite-dimensional, nilpotent Lie algebra,
and let $\g=\sl_2$ or $\g=\sol_2$. Then
$\F(\CC^\hdot (\n), \g)=\F^1(\CC^\hdot (\n), \g)$.
 \end{lemma}

\begin{proof}
In view of Proposition \ref{prop:flathol} and Example \ref{ex:holc}, we 
need to show that any representation $\rho \colon \n \to \g$ satisfies
$\dim (\im \rho) \leq 1$. Pick generators $a_1, \dots, a_n$
for $\n$.  Nilpotency of $\n$ guarantees the existence of an integer $k$
such that all length $k$ brackets of $\rho(a_1), \dots, \rho(a_n)$
vanish in $\g$.

Recall from Remark \ref{rem:comm} that $[b, b']=0$ in $\g$ if and only if
$\rank \{b, b'\} \leq 1$. Assuming $k>2$, we claim that the above
vanishing property also holds for $k-1$. Putting together these two
facts, it follows that $\im(\rho)$ is at most $1$-dimensional, as needed.

To verify our claim, let us assume that there is a non-trivial bracket $b$
on $\rho(a_1), \dots , \rho(a_n)$ of length $k-1$. By our
assumption, $[\rho(a_i), b]=0$, for all $i$. Hence, $\rho(a_i) \in \C \cdot b$,
for all $i$, and consequently $b=0$, since $k >2$. This contradiction
completes the proof.
\end{proof}

Note that the conclusions of the lemma may fail if $\n$ is
only assumed to be solvable.  For instance, if $\n=\sol_2$,
then $\F(\CC^{\hdot}(\n), \n)= \Hom_{\Lie}(\n, \n)$
contains the identity map.

\begin{theorem}
\label{thm:resnilp}
Let $\pi$ be a finitely generated, nilpotent group, with Malcev--Lie
algebra $\n$, and let $V_{\theta}$
be a finite-dimen\-sional $\g$-module, where $\g=\sl_2$ or $\sol_2$. Then
\[
\RR_1^k (\CC^{\hdot}(\n), \theta)=
\begin{cases}
\emptyset &\text{if $k > \dim \n$},\\
\Pi(\CC^\hdot(\n), \theta)\ne \emptyset
&\text{if $k \leq \dim \n$}.
\end{cases}
\]
\end{theorem}

\begin{proof}
Set $A^\hdot = \CC^\hdot(\n)$. 
The first assertion is now immediate,  since $\CC^k (\n)=0$ for $k > \dim \n$.

Using the identification from Corollary \ref{cor:liecoh} in the case
when $\theta=\id_{\C}$, work of Dixmier \cite{Di} shows that
$\RR_1^k (A)\subseteq \{0\}$ for all $k$, by \cite[Theorem 1]{Di},
and $ 0 \in \RR_1^k (A)$ for $k \leq \dim \n$,
by \cite[Theorem 2]{Di}.

Now, by Lemma \ref{lem:nilpflat}, the set $\F(A, \g)$ consists of all elements
of the form $\omega=\eta \otimes b$, with $\eta \in H^1(A)$ and $b \in \g$.
On the other hand, by Corollary \ref{cor:vdet}, $\omega \in \RR_1^k (A, \theta)$
if and only if $\lambda \eta \in \RR_1^k (A)$, for some eigenvalue $\lambda$ of
$\theta(b)$. If $k \leq \dim \n$, this happens precisely when either $\eta=0$ or
$\lambda=0$, i.e., when $\omega \in \Pi(A, \theta)$.
\end{proof}

\begin{remark}
\label{rem:dix}
It is possible to state a more general version of Theorem \ref{thm:resnilp}, 
based on the aforementioned results of Dixmier.  Indeed, let $\f$ be an 
arbitrary finite-dimensional Lie algebra, and recall that we have a canonical
isomorphism $\F(\CC^\hdot (\f), \gl(V)) \isom \Hom_{\Lie}(\f, \gl(V))$, which
identifies a flat connection $\omega$ with the $\f$-module $V_{\omega}$.
Applying Lemma \ref{lem:cohlie} to the representation $\theta=\id_{\gl(V)}$, 
we see that the Aomoto complex $(\CC^\hdot(\f) \otimes V, d_{\omega})$ 
is naturally isomorphic to the Chevalley--Eilenberg complex 
$\CC^\hdot (\f, V_{\omega})$, which computes the Lie algebra 
cohomology of $\f$ with coefficients in $V_{\omega}$. Consequently, 
$\RR_1^\hdot (\CC^\hdot (\f), \id_{\gl(V)})$ may be identified with the 
set of $\f$-module structures on $V$ for which $H ^\hdot (\f, V) \neq 0$.

When $\f=\n$ is a nilpotent Lie algebra, Theorems 1 and 2 from \cite{Di} imply
that $\RR_1^{k}(\CC^\hdot (\n), \id_{\gl(V)})$ consists of those $\n$-modules
$V$ that have a non-zero, $\n$-trivial sub-quotient, provided $k \leq \dim \n$, and
otherwise is empty. However, both the ambient representation variety
$\Hom_{\Lie}(\n, \gl(V))$, and, {\em a fortiori}, the locus of representations 
having non-zero, $\n$-trivial sub-quotients, may be difficult to compute 
explicitly. This is the reason why we chose to state Theorem \ref{thm:resnilp} 
only for the Lie algebras $\g=\sl_2$ and $\g=\sol_2$, where the explicit 
description from Lemma \ref{lem:nilpflat} is available.
\end{remark}

\section{Algebras with linear resonance}
\label{sect:linres}

In this section we continue our analysis of $\g$-resonance varieties, 
focussing our attention on \cdga's with linear resonance and zero differential. 
The guiding examples for this study are the cohomology rings 
of $1$-formal groups, such as the Artin groups.

Let $A^\hdot$ be a connected \cdga, and let
\begin{equation}
\label{eq:res1}
\RR_1^1 (A)=
\{\alpha\in Z^1(A) \mid \text{$\exists \beta \in A^1 \,\setminus\, \C\cdot \alpha$
 such that $d\beta + \alpha \beta=0$ in $A^2$}\}
\end{equation}
be its first resonance variety. Clearly, $\RR_1^1 (A)$ depends
only on the differential $d \colon A^1 \rightarrow A^2$ and on
the multiplication map $\cup \colon A^1 \wedge A^1 \to A^2$;
in fact, as noted in Lemma \ref{lem:cores}, the set $\RR_1^1 (A)$
depends only on the co-restriction of these maps to the subspace
$\im(d)+\im (\cup)\subseteq A^2$.

Let $A^{\leq 2}=A/A^{>2}$ be the degree-$2$ truncation of $A$;
by the above remarks, $\RR_1^1 (A)=\RR_1^1 (A^{\le 2})$.
Note that a connected $\cdga$ $A^ \hdot$ that is concentrated
in degrees at most $2$ is
encoded by a pair of linear maps, $d \colon A^1 \rightarrow A^2$ and
$\cup \colon A^1 \wedge A^1 \to A^2$. Moreover, \cdga~maps in this
subcategory have an obvious linear algebra translation.

\begin{definition}
\label{def:linres}
A connected $\cdga$ $A^\hdot$ has {\em linear resonance}\/
if $\RR_1^1 (A)$ is a finite union of the form
\begin{equation}
\label{eq:linres}
\RR^1_1(A)=\bigcup_{C\in \cC} C,
\end{equation}
where each $C$ is a linear subspace of $H^1(A)=Z^1(A)$.
\end{definition}

As shown in \cite{DPS}, if $\pi$ is a finitely-generated,
$1$-formal group, and $A^\hdot= (H^\hdot (\pi,\C), d=0)$,
then $A^\hdot$ has linear resonance. Moreover, as shown
in \cite{DP}, Gysin models of quasi-projective manifolds
also have this property.

Given a connected $\cdga$ $A^\hdot$ and a linear subspace
$C \subseteq A^1$, let $A^\hdot _C$ be the connected $\cdga$
concentrated in degrees $\leq 2$, defined by the restriction
maps, $d_C \colon C \to A^2$ and $\cup_C \colon C\wedge C \to A^2$.
Clearly, $A^\hdot _C \hookrightarrow A^{\leq 2}$ is a sub-$\cdga$
of the truncated $\cdga$ $A^{\leq 2}$.

\begin{lemma}
\label{lem:flat lin}
Let $A^\hdot$ be a connected $\cdga$ with linear resonance,
with $\RR^1_1(A)$ decomposed as in \eqref{eq:linres}.
Then, for any Lie algebra $\g$,
\begin{equation}
\label{eq:fincl}
\F(A, \g) \supseteq \F^1(A, \g) \cup \bigcup_{C\in\cC}\F(A_C, \g).
\end{equation}
Moreover, if both $A^1$ and $\g$ are finite-dimensional, then
both $\F^1(A, \g)$ and $\F(A_C, \g)$ are Zariski-closed subsets
of $\F(A, \g)$.
\end{lemma}

\begin{proof}
The first assertion follows from the naturality properties of the
parameter set for flat connections.
The second assertion follows from the equality
$\F(A_C, \g)=\F(A, \g) \cap C \otimes \g$.
\end{proof}

\begin{prop}
\label{prop:feq}
Suppose $A^\hdot$ is a $\cdga$ with zero differential and linear
resonance.  Then, if $\g =\sl_2$ or $\sol_2$, inclusion \eqref{eq:fincl}
becomes an equality.
\end{prop}

\begin{proof}
An element $\omega \in A^1 \otimes \g$ may be written as
$\omega= \alpha \otimes h + \beta \otimes x$ when $\g=\sol_2$, and
$\omega= \alpha \otimes H + \beta \otimes X+ \gamma \otimes Y$
when $\g=\sl_2$. Let $U$ be the linear subspace of $A^1$ generated
by $\alpha, \beta$ in the first case, and by $\alpha, \beta, \gamma$
in the second case.  In either case, an easy computation
shows that $\omega$ satisfies the flatness condition \eqref{eq:flat}
if and only if $U$ is an isotropic subspace with respect to
$\cup \colon A^1\wedge A^1 \to A^2$.

Suppose $\omega \notin \F^1(A, \g)$, that is, $\dim U \geq 2$. Then
clearly $U \subseteq \RR_1^1(A)$. Now, by assumption,
each irreducible component of $\RR^1_1(A)$ is a linear 
subspace of $H^1(A)=A^1$. Hence, $U$ is contained in one of 
those linear subspaces, say, $U \subseteq C$. Therefore, 
$\omega \in \F(A, \g) \cap C\otimes \g=\F(A_C, \g)$,
and this completes the proof.
\end{proof}

\begin{theorem}
\label{thm:rincl}
Let $A^\hdot$ be a $\cdga$ with zero differential and linear
resonance, and assume that $A^1 \neq 0$. Let $\g=\sl_2$ or $\sol_2$,
and let $\theta \colon \g \to \gl(V)$ be a finite-dimensional representation.
If $\RR^1_1(A)$ decomposes as in \eqref{eq:linres}, then
\begin{equation}
\label{eq:r11atheta}
\RR_1^1(A, \theta)=\Pi(A, \theta) \cup \bigcup_{C\in \cC} \F(A_C, \g).
\end{equation}
\end{theorem}

\begin{proof}
First we prove that inclusion $\supseteq$ holds. 
In view of Corollary \ref{cor:vdet}, we only need to show that
$ \F(A_C, \g) \subseteq \RR_1^1(A, \theta)$, for all $C\in \cC$.
To that end, let $\omega \in C\otimes \g$ be an arbitrary element satisfying the
Maurer--Cartan equation \eqref{eq:flat}. The argument from the proof of
Proposition  \ref{prop:feq} implies that $\omega \in U \otimes \g$, where
$U \subseteq A^1$ is an isotropic subspace. Using formula \eqref{eq:adw},
we see that $d_{\omega}(U \otimes V)=0$. If $\omega$ is regular, i.e.,
$\dim U \geq 2$, then clearly $\omega \in \RR_1^1(A, \theta)$, since
$\dim d_\omega (A^0 \otimes V) \leq \dim V$.
Otherwise, $\omega$ must be of the form $\eta \otimes g$, for some
$\eta \in C$.  Now, by assumption, $C \subseteq \RR_1^1(A)$. Hence,
by Corollary \ref{cor:vdet}, we again conclude that
$\omega \in \RR_1^1(A, \theta)$.

Next, we prove that inclusion $\subseteq$ holds. 
To that end, let $\omega \in \RR_1^1(A, \theta)$, and suppose
$\omega \notin \bigcup_{C \in \cC} \F(A_C, \g)$.  Then, by Proposition \ref{prop:feq},
$\omega$ is of the form $\eta \otimes g$.
By Corollary \ref{cor:vdet}, there is an eigenvalue $\lambda$ of $\theta(g)$
such that $\lambda \eta \in \RR_1^1(A)$. If $\lambda \neq 0$, our linearity
of resonance assumption implies that $\eta \in C$ for some $C\in \cC$,
a contradiction. Hence, $\omega \in \Pi(A, \theta)$, and we are done.
\end{proof}

In the important case of $1$-formal groups, Proposition \ref{prop:feq}
and Theorem \ref{thm:rincl} imply that the decomposition into irreducible
components of the usual rank $1$ resonance varieties in degree and
depth $1$ (guaranteed by \cite{DPS}) determines the structure of both
$\g$-valued flat connections and higher-rank resonance varieties, in the
same degree and depth, provided $\g=\sl_2$ or $\sol_2$.  We will illustrate 
this phenomenon in the next two sections.

\section{Artin groups}
\label{sect:raag}

We now turn our attention to a class of groups which plays a prominent 
role in geometric group theory and low-dimensional topology.  For most 
of the section we will work with the restricted class of right-angled Artin groups, 
which are the most amenable to explicit computations in our setting, and then 
we will return towards the end to the wider class of Artin groups. 

\subsection{Rank $1$ resonance}
\label{subsec:rk1 raag}
Let $\G=(\V,\E, \ell)$ be a labeled, finite simplicial graph, with
vertex set $\V$, edge set $\E \subseteq {\V \choose 2}$  and labeling
function $\ell \colon \E \to \Z_{\ge 2}$.  The {\em Artin group}\/ associated 
to such a graph is the group $\pi_{\G,\ell}$, with a generator $v$ for each 
vertex $v \in \V$, and a defining relation of the form $vwv\dots=wvw\cdots$
for each edge $e=\{v,w\}$ in $\E$, where the words on each side of
the equality are of length $\ell(e)$.  As shown in \cite{KM}, Artin groups
are $1$-formal.

An (unlabeled) finite simplicial graph $\G=(\V,\E)$ may be viewed
as a labeled  graph by setting $\ell(e)=2$, for each $e \in \E$.
The corresponding group is called the {\em right-angled
Artin group}\/ (for short, \raag) associated to $\G$, and
is simply denoted by $\pi_\G$.

As is well-known, the cohomology algebra 
$A^{\hdot}_{\G}=H^{\hdot}(\pi_{\G}, \C)$ is isomorphic to 
the exterior Stanley--Reisner ring of the graph $\G$.  That is
to say, $A^{\hdot}_{\G}$ is the quotient of the exterior $\C$-algebra
on generators $v^*$, for all $v\in \V$, by the ideal generated by
the monomials $v^*w^*$, for all $\{v,w\}\notin \E$.  Let us endow
this algebra with the differential $d=0$. The corresponding
holonomy Lie algebra, $\h(\G):=\h(A^\hdot_{\G})$, has presentation
\begin{equation}
\label{eq:holo artin}
\h(\G)=\L( \V) / ( [v,w]=0\ \text{if $\{v,w\}\in \E$}).
\end{equation}

Since the group $\pi_{\G}$ is $1$-formal, we know that the cohomology
algebra $A^\hdot_{\G}$ is a \cdga~with linear resonance,
that is, the resonance variety $\RR_1^1(A^\hdot_{\G})$ decomposes
as a finite union of linear subspaces of $A^1_{\G}=\C^{\V}$. This
decomposition was worked out explicitly in \cite[Theorem 5.5]{PS-artin},
as follows:
\begin{equation}
\label{eq:r11 ag}
\RR_1^1(A^\hdot_{\G})= \bigcup_{\substack{\W\subseteq \V\\ c(\W)>1}}\C^\W,
\end{equation}
where $c(\W)$ denotes the number of connected components of the
induced subgraph $\G_\W$, and $\C^{\W}$ denotes the coordinate
subspace of $\C^{\V}$ spanned by $\W$.

\subsection{Flat connections}
\label{subsec:flat raag}
We now analyze the space of flat $\g$-connections on the
algebra $A=A^{\hdot}_{\G}$ endowed with the zero differential,
with the goal of making explicit the decomposition from
Proposition \ref{prop:feq}.

Given a Lie algebra $\g$, we will use the identification
$\F(A,\g) \cong \Hom_{\Lie} (\h(\G), \g)$ induced by the natural 
isomorphism $\C^\V \otimes \g \cong \Hom (\C \langle \V \rangle, \g)$, 
as in Proposition \ref{prop:flathol}, and we will write elements
$\omega \in \C^\V \otimes \g$ as tuples of elements
$\omega_v \in \g$, indexed by the vertices $v \in \V$.
In view of formula \eqref{eq:holo artin}, $\omega$ belongs
to $\F(A, \g)$ if and only if
$[\omega_u, \omega_v]=0$ whenever $\{u,v\} \in \E$.

For each subset $\W \subseteq \V$, let $\W_1,\dots,\W_c$ be the
connected components of the vertex set of $\G_\W$, let
$\overline{\W}=\V\setminus \W$, and put
\begin{equation}
\label{eq:sw}
S_\W = \left\{
\omega \in \C^\V \otimes \g \left|
\begin{array}{ll}
\omega_v=0 &  \text{for $v \in \overline{\W}$}
 \\
\rank\{\omega_v\}_{v \in \W_i} \leq 1 & \text{for $1\le i \le c$}
\end{array}
\right. \right\}.
\end{equation}

If $\g$ is a finite-dimensional Lie algebra, then clearly $S_{\W}$ is
a Zariski-closed subset of the affine space
$\C^\W \otimes \g \subseteq \C^\V \otimes \g$. 
More specifically, 
\begin{equation}
\label{eq:prodcone}
S_{\W} \cong \prod_{i=1}^{c}
\operatorname{cone} \left( \PP(\C^{\W_i})\times \PP(\g) \right).
\end{equation}

\begin{lemma}
\label{lem:sw}
For every Lie algebra $\g$, and every subset $\W\subseteq \V$, the following hold.
\begin{enumerate}
\item \label{w1}
If $c(\W)=1$, then $S_{\W}\subseteq \F^1(A_{\C^\W}, \g)$.
\item \label{w2}
If $c(\W)>1$, then $S_{\W}\subseteq \F(A_{\C^\W}, \g)$.
\end{enumerate}
\end{lemma}

\begin{proof}
Part \eqref{w1} is immediate from the definitions.

To prove part \eqref{w2}, let $\omega \in S_{\W}$.
Let $\{u,v\}$ be an edge in $\E$. If $\{u,v\} \nsubseteq \W$, then
either $\omega_u=0$ or $\omega_v=0$, and so $[\omega_u, \omega_v]=0$.
Otherwise, $\{u,v\} \subseteq \W_i$ for some $1\le i\le c$; thus,
$\rank\{\omega_u,\omega_v\}\leq 1$, and so again
$[\omega_u, \omega_v]=0$. Therefore, $\omega \in \Hom_{\Lie}(\h(\G), \g)$,
and the desired conclusion follows.
\end{proof}

\begin{prop}
\label{prop:raagf}
Let $\G=(\V,\E)$ be a finite simplicial graph, and let $\g$ be a Lie algebra.
Then $\F(A^\hdot_\G, \g)\supseteq \bigcup_{\W \subseteq \V} S_{\W}$.
Moreover, if $\g = \sl_2$ or $\sol_2$, then this inclusion holds as equality,
\begin{equation}
\label{eq:flat raag}
\F(A^\hdot_\G, \g)= \bigcup_{\W \subseteq \V} S_{\W}.
\end{equation}
\end{prop}

\begin{proof}
The first claim follows at once from Lemmas \ref{lem:flat lin}
and \ref{lem:sw}.

To prove the second claim, let $\omega \in \C^\V\otimes \g$ be an
arbitrary element, let $\supp(\omega)=\{v \in \V \mid \omega_v \neq 0\}$
be its support, and let $\V=\W_1 \coprod  \cdots \coprod \W_c\coprod \overline{\W}$
be the partition associated to $\W=\supp(\omega)$.
By construction, $\omega_v=0$ for all $v \in \overline{\W}$.
For each $1\le i\le c$ and each $u,v \in \W_i$, consider an
edge-path in $\G_{\W_i}$ connecting $u$ to $v$. If
$\omega \in \Hom_{\Lie}(\h(\G), \g)$, it follows from Remark \ref{rem:comm}
that $\rank\{\omega_u,\omega_v\}\le 1$. Hence, $\omega \in S_{\W}$.
\end{proof}

\subsection{Higher-rank resonance}
\label{subsec:higher raag}
Let $\theta\colon \g\to \gl(V)$ be a finite-dimensional representation
of a finite-dimensional Lie algebra $\g$. We now describe a similar
formula for the higher-rank resonance varieties
$\RR_1^1(A^\hdot_\G, \theta)$,
which makes explicit the decomposition from Theorem \ref{thm:rincl}.

Given a subset $\W \subseteq \V$, put
\begin{equation}
\label{eq:pw}
P_\W = \left\{\omega\in \C^{\V} \otimes \g \left|
\begin{array}{ll}
\omega_v=0  & \text{if $v \in \overline{\W}$}
\\
\text{$\omega_v=\lambda_v g_{\W}$, where $\lambda_v \in \C$ and
$g_{\W} \in V(\det \circ  \theta)$}
& \text{if $v \in \W$}
\end{array}
\!\!\right.\right\}.
\end{equation}

Clearly, $P_{\W}$ is a Zariski-closed subset of $\C^{\V} \otimes \g$.
Since the element $g_{\W} \in V(\det \circ  \theta)$ is
independent of the vertex $v\in \W$, we have that
$\rank\{\omega_v\}_{v \in \W} \leq 1$,
for every $\omega\in P_\W$.  Thus, $P_\W\subseteq S_\W$.

\begin{theorem}
\label{thm:raagr}
Let $\g=\sl_2$ or $\sol_2$.  If $\G=(\V,\E)$ is a finite, simplicial graph,
and $V_{\theta}$ is a finite-dimensional $\g$-module, then
\begin{equation}
\label{eq:res raag}
\RR_1^1(A^{\hdot}_\G, \theta)=
\bigcup_{\substack{\W\subseteq \V\\ c(\W)=1}}P_\W \cup
\bigcup_{\substack{\W\subseteq \V\\ c(\W)>1}}S_\W.
\end{equation}
\end{theorem}

\begin{proof}
Let $A=A^{\hdot}_\G$, equipped as before with the zero
differential.  In view of Theorem \ref{thm:rincl}, the right-hand side of
\eqref{eq:res raag} is  included in $\RR_1^1(A, \theta)$.  Indeed, we
have that $P_\W \subseteq \Pi(A, \theta)$, for any subset
$\W \subseteq \V$, by construction.  Moreover,
$S_\W \subseteq \F(A_{\C^\W}, \g)$, for any subset $\W \subseteq \V$
with $c(\W)>1$, by Lemma \ref{lem:sw}.

To prove the reverse inclusion, let $\omega$ be a nonzero element in
$\RR_1^1 (A, \theta) \subseteq \F(A, \g)$. By Proposition \ref{prop:raagf}, 
we have that $\omega \in S_\W$, for some $\W\subseteq \V$.
If $\omega$ is regular, then necessarily $c(\W)>1$, since otherwise 
$S_\W \subseteq \F^1(A, \g)$, by  Lemma \ref{lem:sw}. Thus, in order 
to finish the proof, we must show the following:
\begin{equation}
\label{eq:implies}
\omega \in \F^1(A, \g) \setminus \bigcup_{c(\W)>1}S_\W \  \Rightarrow \ 
\omega \in \bigcup_{c(\W)=1}P_\W.
\end{equation}

Set $\W=\supp(\omega)$. As we saw in the proof of Proposition \ref{prop:raagf},
$\omega \in S_\W$, and so the induced subgraph $\G_\W$ is connected. 
Theorem \ref{thm:rincl} tells us that either $\omega \in \Pi (A, \theta)$, or 
$\omega \in P(\C^{\W'} \times \g)$, for some subset $\W'\subseteq \V$ with 
$c(\W')>1$.  In the first case, it follows that $\omega \in P_{\W}$, and we 
are done. In the second case, we must have $\W \subseteq \W'$, and so 
$\W \subseteq \W'_i$, for some $1\le i \le c(\G_{\W'})$. Since 
$\omega \in \F^1(A, \g)$ and $\W=\supp(\omega)$, this implies
$\omega \in S_{\W'}$, a contradiction.
\end{proof}

\subsection{Irreducible components}
\label{subsec:irred}
In the case when $\g=\sl_2$, we may use Proposition \ref{prop:raagf} and
Theorem \ref{thm:raagr} to find the decompositions of $\F(A^\hdot_\G, \g)$
and $\RR^1_1(A^\hdot_\G, \theta)$ into irreducible components.

To that end, we need to introduce a partial order among the subsets
of the vertex set $\V$.  For each inclusion $\W \subseteq \W'$
of vertex subsets, let
$K_{\W \W'} \colon \{ \W_1, \dots, \W_c \}\to \{ \W'_1, \dots, \W'_{c'} \}$
be the induced map from the connected components of $\G_{\W}$
to the connected components of $\G_{\W'}$.  The order relation on
the subsets of $\V$ is then given by
\begin{equation}
\label{eq:prec}
\text{$\W \preccurlyeq  \W' \Leftrightarrow \W \subseteq \W'$
and $K_{\W \W'}$ is injective.}
\end{equation}
Clearly, if $c(\W) > 1$ and $c(\W')=1$, then $\W \npreccurlyeq \W'$. Furthermore,
if $c(\W)=1$, then $\W \preccurlyeq \W'$ if and only if $\W \subseteq \W'$.

\begin{lemma}
\label{lem:cases}
With notation as above, the following hold.
\begin{enumerate}
\item \label{s1}
$S_\W \subseteq S_{\W'} \Leftrightarrow \W  \preccurlyeq \W'$,
for all $\W, \W'$.
\item \label{s2}
$S_\W \nsubseteq P_{\W'}$, provided $c(\W)>1$ and $c(\W')=1$.
\item \label{s3}
$P_\W \subseteq P_{\W'} \Leftrightarrow \W \preccurlyeq \W'$, provided
$c(\W)=c(\W')=1$.
\item \label{s4}
$P_{\W} \subseteq S_{\W'} \Leftrightarrow \W  \preccurlyeq \W'$, provided
$c(\W)=1$ and $c(\W')>1$.
\end{enumerate}
\end{lemma}

\begin{proof}
We start with part \eqref{s1}.
The implication $\W \preccurlyeq \W' \Rightarrow S_{\W} \subseteq S_{\W'}$
is immediate. Conversely, suppose $S_{\W} \subseteq S_{\W'}$.
Then clearly $\W \subseteq \W'$, by a support argument.
It remains to show that $K_{\W\W'}$ is injective.
If that were not the case, there would be two distinct connected
components of $\G_{\W}$ such that their vertex sets, call them
$\W_1$ and $\W_2$,  are both included in a component $\W'_1$ of $\G_{\W'}$.
Now pick $v_i \in \W_i$ and $g_i \in \g$ such that $\rank \{ g_1, g_2 \}=2$,
and define an element $\omega\in \C^{\V}\otimes \g$ by setting
$\omega_{v_i}=g_i$ and $\omega_v=0$, otherwise.
We then have $\omega \in S_{\W} \setminus S_{\W'}$,
a contradiction.

To prove part \eqref{s2}, suppose that $S_{\W} \subseteq P_{\W'}$.
Then, since $P_{\W'} \subseteq S_{\W'}$, we must have
$S_{\W} \subseteq S_{\W'}$. By part \eqref{s1}, this implies
$\W \preccurlyeq \W'$, and so $K_{\W\W'}$ must be injective,
contradicting the assumption that $c(\W)>c(\W')$.

Before proceeding with the last two parts, recall from Lemma \ref{lem:det}
that there are two possibilities for the variety $V(\det \circ \theta)$:
it is either equal to $\g$, or to $V(\det)$.
In the first case, we have that $P_\W=S_\W$, for all $\W\subseteq \V$
with $c(\W)=1$, and so parts \eqref{s3} and \eqref{s4} follow at once
from part \eqref{s1}.
In the second case, let us pick a non-zero matrix $g \in V(\det)$,
and prove parts \eqref{s3} and \eqref{s4} separately.

For part \eqref{s3}, suppose that $\W \subseteq \W'$. Since we
are assuming that $c(\W)=c(\W')=1$, we must have
$P_\W \subseteq P_{\W'}$, by part \eqref{s1}.
Conversely, suppose $\W \not\subseteq \W'$.
Then pick $v \in \W \setminus \W'$, and define an element
$\omega\in \C^{\V}\otimes \g$ by setting $\omega_v=g$ and $\omega_u=0$
for $u \neq v$. Clearly, $\omega \in P_\W \setminus P_{\W'}$,
and so $P_\W \not\subseteq P_{\W'}$.

For part \eqref{s4}, suppose that $\W \subseteq \W'$.  Then,
since $c(\W)=1$, we must have $P_\W \subseteq S_{\W'}$,
again by part \eqref{s1}.   Conversely, suppose $v\in \W \setminus \W'$.
If we set $\omega_v=g$ and $\omega_u=0$ for $u \neq v$, we then
have $\omega \in  P_\W \setminus S_{\W'}$, and this finishes the proof.
\end{proof}

\begin{corollary}
\label{cor:irred}
Let $\G$ be a finite, simplicial graph, and let $\theta\colon \sl_2\to \gl(V)$ be a
finite-dimensional representation. We then have the following decompositions into
irreducible components:
\begin{align}
\label{eq:irred raag}
\F(A^\hdot_\G, \sl_2)&=
\bigcup_{\text{$\W$ $\preccurlyeq$-maximal}} S_{\W},
\\
\RR_1^1(A^{\hdot}_\G, \theta)&=
\bigcup_{\substack{c(\W)=1\\ \text{$\W$ $\preccurlyeq$-maximal}}} P_\W
\cup  \bigcup_{\substack{c(\W)>1 \\
\text{$\not\exists \W\precneqq\W'$ with $c(\W')>1$}}} S_\W.
\end{align}
\end{corollary}

\begin{proof}
For any subset $\W$ of the vertex set $\V$, both $P_\W$ and $S_\W$
are irreducible, Zariski-closed subsets of  $\F(A^\hdot_\G, \sl_2)$.
Indeed, the first set equals $P(\C ^\W \times V(\det \circ \theta))$,
which is irreducible by Lemma \ref{lem:det}, while the second
set equals $\prod_{i=1}^{c(\W)} P(\C^{\W_i} \times \sl_2)$, which
is clearly irreducible.
Using Lemma \ref{lem:cases} to eliminate redundancies
in \eqref{eq:flat raag} and \eqref{eq:res raag} completes the proof.
\end{proof}

\subsection{Semisimple Lie algebras}
\label{subsec:limits}
Our emphasis on the $\sl_2$ case comes from the fact that the
inclusion from Proposition \ref{prop:raagf} may be strict
for arbitrary Lie algebras $\g$ and graphs $\G$. More precisely,
we have the following result.

\begin{prop}
\label{prop:semi}
Suppose $\g$ is a semisimple Lie algebra, different from $\sl_2$.
There is then a connected, finite simple graph $\G$ such that
$\F(A^\hdot_\G, \g)\ne \bigcup_{\W \subseteq \V} S_{\W}$.
\end{prop}

\begin{proof}
For each such $\g$, we will construct a connected, finite simple
graph $\G$ and a regular
element $\omega \in \F(A^{\hdot}_{\G}, \g)$ with the property that
$\omega _v \neq 0$ for all $v\in \V$. This will
imply  $\omega \notin \bigcup_{\W \subseteq \V} S_{\W}$.
Indeed, if $\omega \in S_{\W}$, then necessarily $\W=\V$, since
$\supp(\omega)=\V$; see Lemma \ref{lem:sw}. More precisely, it follows from
Lemma \ref{lem:sw}\eqref{w1} that $\omega \in \F^1(A^{\hdot}_{\G}, \g)$,
contradicting regularity.

First, we need to recall some standard facts on
semisimple Lie algebras, cf.~\cite{Hu}. Let $r$ be the dimension
of a Cartan subalgebra of $\g$; the assumption that $\g\ne \sl_2$
means that $r>1$. Pick a system of simple roots $\{\alpha_1, \dots, \alpha_r\}$.
For an arbitrary root $\alpha$, pick a generator $g_\alpha$ of the root space
$\g_\alpha \subseteq \g$. Then the family $\{g_\alpha\}_{\alpha}$ is independent,
$[g_{\alpha_i}, g_{-\alpha_j}]=0$ for $i \neq j$, and $[g_\alpha, g_\beta]=0$
if $0 \neq \alpha + \beta$ is not a root. Choose a positive root $\gamma$
of maximal height. Then $[g_{\gamma}, g_{\alpha_i}]=[g_{-\gamma}, g_{-\alpha_i}]=0$,
for all $1 \leq i \leq r$.

If $r>2$, we let $\G$ be the graph with vertex set
$\V=\{\pm \alpha_1, \dots, \pm\alpha_r\}$ and edges
$\{\alpha_i,-\alpha_j\}$ for $i \neq j$. Clearly, $\G$ is connected.
Now set $\omega_\alpha=g_\alpha$ for $\alpha \in \V$.
Then $\omega\in \F(A^{\hdot}_{\G}, \g)$ is regular and
$\supp(\omega)=\V$, as needed.

If $r=2$ and $\g \neq \sl_2 \times \sl_2$, we let
$\V=\{\pm\alpha_1, \pm\alpha_2, \pm\gamma\}$
and declare the edges to be $\{\alpha_i,-\alpha_j\}$ for $i \neq j$ and
$\{\epsilon \alpha_i, \epsilon \gamma\}$, with $\epsilon = \pm 1$,
while if $\g=\sl_2 \times \sl_2$, we let $\V=\{\pm \alpha_1, \pm \alpha_2\}$
and declare the edges to be $\{\alpha_i,-\alpha_j\}$ for $i \neq j$ and
$\{\epsilon \alpha_1, \epsilon \alpha_2\}$, with $\epsilon = \pm 1$.
In both cases, the resulting graph $\G$ is connected, and the
desired form $\omega$ is constructed as before.
\end{proof}

\subsection{Labeled graphs and odd contractions}
\label{subsec:odd}

We now return to the general case of a (finitely generated) Artin group.  
A construction described in \cite[\S 11.9]{DPS} associates to each 
labeled, finite simplicial graph $(\G,\ell)$ an unlabeled graph, $\tilde{\G}$, 
called the odd contraction of $(\G,\ell)$, as follows.  We first define an 
unlabeled graph $\Gamma_{\rm odd}$ by keeping all the vertices of $\G$, 
and retaining only those edges for which the label is odd (after which the 
label is discarded).  We then let $\tilde{\G}$ be the graph whose vertices 
correspond to the connected components of $\Gamma_{\rm odd}$, with 
two distinct components determining an edge $\{c,c'\}$ in  $\tilde{\G}$ 
if and only if there exist vertices $v\in c$ and $v'\in c'$ which 
are connected by an edge in $\G$.

Let $\pi_{\G, \ell}$ be the Artin group associated to the labeled graph $(\G,\ell)$, 
and let $\pi_{\tilde{\G}}$ be the right-angled Artin group associated to the odd 
contraction of $(\G,\ell)$.  Furthermore, let $A^{\hdot}_{\G,\ell}=
H^\hdot (\pi_{\G, \ell},\C)$ and $A^{\hdot}_{\tilde{\G}}=H^\hdot (\pi_{\tilde{\G}},\C)$ 
be the respective cohomology algebras, both endowed with the zero differential. 

\begin{prop}
\label{prop:gamma ell}
For each finite-dimensional Lie algebra $\g$ and each 
finite-dimensional representation $\theta\colon \g\to \gl(V)$, 
there is an isomorphism
\[
(\F(A^\hdot_{\G,\ell}, \g), \RR_1^1(A^{\hdot}_{\G,\ell}, \theta))
\cong (\F(A^\hdot_{\tilde{\G}}, \g), \RR_1^1(A^{\hdot}_{\tilde{\G}}, \theta))
\]
between the corresponding pairs of affine varieties.
\end{prop}

\begin{proof}
As shown in \cite[Lemma  11.11]{DPS}, the groups $\pi_{\G, \ell}$ and 
$\pi_{\tilde{\G}}$ have the same Malcev--Lie algebra, and thus the 
same Sullivan $1$-minimal model. In particular, the \cdga's 
$A^{\hdot}_{\G,\ell}$ and $A^{\hdot}_{\tilde{\G}}$ have the same 
co-restrictions in low degrees.  Lemma \ref{lem:cores}, then, yields 
the desired isomorphism.
\end{proof}

When combined with Corollary \ref{cor:irred}, this proposition produces  
explicit decompositions into irreducible components for the varieties 
$\F(A^\hdot_{\G,\ell}, \sl_2)$ and $\RR_1^1(A^{\hdot}_{\G,\ell}, \theta)$, 
for any labeled, finite simplicial graph $(\G,\ell)$ and any finite-dimensional 
representation $\theta\colon \sl_2\to \gl(V)$.

\subsection{Germs of representation varieties}
\label{subsec:germs_rep}

To conclude this section, let us compare our computations
with the Kapovich--Millson universality theorem \cite{KM}.     
This striking result (Theorem 1.9 from \cite{KM}) 
says the following: Given a point $x$ on an
affine variety $\XX$ defined  over $\Q$, there is a labeled
finite simplicial graph $(\G, \ell)$ and a (nontrivial) 
representation $\rho \colon \pi_{\G, \ell} \to \PSL_2$ 
with finite image and trivial centralizer 
such that the germ of the GIT 
quotient of the representation variety 
at the class of $\rho$ is isomorphic to the  germ of the given 
singularity; that is, 
\begin{equation}
\label{eq:git}
\big(\Hom (\pi_{\G, \ell}, \PSL_2)/\!/\PSL_2 \big)_{([\rho])} 
\cong \XX_{(x)}.
\end{equation}
Moreover, Kapovich and Millson show in the proof  of 
\cite[Theorem 14.7]{KM} that the representation variety admits 
a local cross-section for the conjugation action 
of $\PSL_2$ at the point $\rho$.  Therefore, 
\begin{equation}
\label{eq:isom km}
\Hom (\pi_{\G, \ell}, \PSL_2)_{(\rho)} \cong \XX_{(x)} \times \C^3_{(0)},
\end{equation}
as analytic germs. 

So what can be said at  the trivial representation?  
By Theorem 17.3 from \cite{KM}, 
the  variety $\Hom (\pi_{\G, \ell}, \PSL_2)$ has at worst 
a quadratic singularity at the origin $\rho=1$.  Our work 
complements this remarkable result, as follows.

Let $\tilde{\G}$ be the odd contraction of $(\G,\ell)$.  It follows from 
Theorem \ref{thm:b1} and Proposition \ref{prop:gamma ell}
that we have a local analytic isomorphism
\begin{equation}
\label{eq:isom sl2}
\Hom (\pi_{\G, \ell}, \PSL_2)_{(1)} \cong
\F(H^\hdot (\pi_{\tilde{\G}},\C), \sl_2)_{(0)}
\end{equation}
which identifies  $\VV_1^1(K(\pi_{\G,\ell},1),\iota)_{(1)}$ with
$\RR_1^1(H^\hdot (\pi_{\tilde{{\G}}},\C), \theta)_{(0)}$, for  every rational
representation $\iota \colon \PSL_2 \to \GL(V)$.  Proposition \ref{prop:raagf} 
then says that the analytic singularity at $1$ of the representation variety 
$\Hom (\pi_{\G, \ell}, \PSL_2)$ can be  completely described in terms 
of the graph $\tilde{\G}$.
Moreover, according to Theorem \ref{thm:raagr}, the same thing happens
with the nonabelian characteristic varieties $\VV^1_1(K(\pi_{\G,\ell},1),\iota)$,
which can be completely described in terms of the graph $\tilde{\G}$ and the 
tangential representation of $\iota$.

\section{Quasi-projective manifolds}
\label{sect:qp}

We devote this last section to the $\g$-resonance varieties
associated to an irreducible, smooth, quasi-projective variety $X$.  
For the most part, we emphasize the case when $X$ is $1$-formal 
and $\g=\sl_2$, a case in which our methods lead to a rather explicit 
description of the variety of $\g$-valued flat connections and of the 
higher-rank resonance varieties, in degree $1$ and depth $1$.

\subsection{Admissible maps onto curves}
\label{subsec:pencils}
Let $X$ be a quasi-projective manifold.
Up to reparametri\-zation at the target, the variety $X$ admits finitely
many {\it admissible}\/ maps $f \colon X \to S$ onto connected,
smooth complex curves with $\chi(S)<0$.  Admissibility means that $f$
is regular, and has connected generic fiber. As shown by Arapura
in \cite{Ar}, the correspondence $f \mapsto f^*(\TT (\pi_1(S)))$ establishes
a bijection between the set $\mathcal{E}_X$ of equivalence classes
of admissible maps from $X$ to a curve $S$ with $\chi(S)<0$ and the set of
positive-dimensional, irreducible components of $\VV^1_1(X)$ containing $1$.

Next, we recall from \cite[Theorem C and Example 5.3]{DP} the relationship
between $\mathcal{E}_X$ and the linear decomposition of $\RR^1_1(A)$, for 
a convenient choice of Gysin model $A^\hdot$ of $X$.
Each admissible map $f$  induces an injective $\cdga$ map
$f^* \colon A^\hdot (S) \hookrightarrow A^\hdot$, where
$A^\hdot (S)$ is the canonical Gysin model of the
curve $S$ and has the property that $A^{>2} (S)=0$. 

Due to the equivalences detailed in \eqref{eq:basept}, we will assume
from now on that $b_1(X)>0$.  In this case, the (linear) irreducible decomposition
of $\RR_1^1(A)$ is given by
\begin{equation}
\label{eq:hpencils}
\RR_1^1(A)= \{ 0\} \cup \bigcup_{f\in \mathcal{E}_X} f^*(H^1(A(S))),
\end{equation}
with the convention that $\{ 0\}$ is omitted when $\mathcal{E}_X \neq \emptyset$.

\subsection{Flat connections and resonance varieties}
\label{subsec:fla qp}
Now let $\g$ be a Lie algebra. Given an admissible map
$f\colon X\to S$, let us denote by $f^{!}$ the linear map
$f^1 \otimes \id_\g\colon A^1 (S) \otimes \g \to A^1 \otimes \g$.
Since the homomorphism $f^* \colon A^\hdot (S) \to A^\hdot$
is injective, the map $f^{!}$ is also injective. By naturality of the
parameter space for flat $\g$-connections,
the map $f^{!}$ restricts to an injective map
$f^{!} \colon \F(A^\hdot (S),\g) \hookrightarrow \F(A^\hdot,\g)$.
Furthermore,
\begin{equation}
\label{eq:foia}
f^{!}(\F(A (S), \g))=\F(A, \g) \cap f^*(A^1 (S)) \otimes \g .
\end{equation}

Let $\theta\colon \g\to \gl(V)$
be a finite-dimensional representation. We then have the following
analog of Proposition \ref{prop:feq} and Theorem \ref{thm:rincl}.

\begin{prop}
\label{prop:flatres qp}
Let $X$ be a quasi-projective manifold, and let $A^\hdot$ be a
convenient Gysin model for $X$, as above. Then,
the following inclusions hold:
\begin{align}
\label{eq:inclf}
\F(A, \g) &\supseteq \F^1(A, \g) \cup
\bigcup_{f\in \mathcal{E}_X} f^{!}(\F(A (S), \g)),
\\
\label{eq:inclr}
\RR^1_1(A, \theta) &\supseteq \Pi (A, \theta)\cup
\bigcup_{f\in \mathcal{E}_X} f^{!}(\F(A (S), \g)),
\end{align}
where, in both instances, $f$ runs through the set
of equivalence classes of admissible maps from $X$ to curves $S$
with $\chi(S)<0$.
Moreover, if $\theta=\id_\C$, then both these inclusions
become equalities.
\end{prop}

\begin{proof}
Inclusion \eqref{eq:inclf} is immediate, while inclusion \eqref{eq:inclr}
follows from Corollary \ref{cor:vdet} and Lemma \ref{lem:functr}(\ref{f2}),
combined with Proposition \ref{prop:chi}.

Finally, assume that $\theta=\id_\C$. Equality in \eqref{eq:inclf} is
then obvious, since $\F(A, \C)=\F^1(A, \C)$, while equality in \eqref{eq:inclr} 
is equivalent to equality \eqref{eq:hpencils}, which we know holds.
\end{proof}

\subsection{The $1$-formal situation}
\label{subsec:1fqp}
Now consider the case when $X$ is a $1$-formal, quasi-projective
manifold, for instance, a smooth, projective variety, or the complement
of a hypersurface in a complex projective space. In this case, Gysin models
may be efficiently replaced by cohomology rings endowed with the zero 
differential, as discussed in Example \ref{ex:formal}.

It follows from \cite[Theorem C and Proposition 7.2(3)]{DPS} that 
the cohomology algebra of a $1$-formal, quasi-projective manifold $X$ 
has linear resonance, in the sense of Definition \ref{def:linres}.  More precisely, 
the decomposition of $\RR^1_1(X)$ into (linear) irreducible components is 
given by
\begin{equation}
\label{eq:resx}
\RR^1_1(X)=\{0\} \cup \bigcup_{f\in \mathcal{E}_X}f^*(H^1(S,\C)),
\end{equation}
again with the convention that $\{0\}$ is omitted when 
$\mathcal{E}_X \neq \emptyset$.  Moreover, the induced 
homomorphism, $f^* \colon H^\hdot (S,\C) \to H^\hdot (X,\C)$,
is an embedding of cohomology rings, and $H^{>2} (S,\C)=0$.

\begin{corollary}
\label{cor:highres qp}
Let $X$ be a $1$-formal, quasi-projective manifold with $b_1(X)>0$.
Let $\g=\sl_2$ or $\sol_2$, and let $\theta \colon \g \to \gl(V)$ be a
finite-dimensional representation. Then, the following equalities hold:
\begin{align}
\label{eq:feq}
\F(H^\hdot (X,\C), \g) &= \F^1(H^\hdot (X,\C), \g) \cup
\bigcup_{f\in \mathcal{E}_X} f^{!}(\F(H^\hdot (S,\C), \g)),
\\
\label{eq:reseq}
\RR^1_1(H^\hdot (X,\C), \theta) &= \Pi (H^\hdot (X,\C), \theta)\cup
\bigcup_{f\in \mathcal{E}_X} f^{!}(\F(H^\hdot (S,\C), \g)),
\end{align}
where, in both instances, $f$ runs through the set
of equivalence classes of admissible maps from $X$ to curves $S$
with $\chi(S)<0$.
\end{corollary}

\begin{proof}
For an admissible function $f\colon X\to S$, write $C=f^*(H^1(S,\C))$.
As noted in the proof of Lemma \ref{lem:flat lin}, the set $\F(H^\hdot (X,\C)_C, \g)$ 
coincides with $\F(H^\hdot (X,\C), \g) \cap C\otimes \g$.  In turn, 
the latter set coincides with $f^{!}(\F(H^\hdot (S,\C), \g))$, since 
$f^* \colon H^\hdot (S,\C) \to H^\hdot (X,\C)$ is injective.  
The result now follows from Proposition \ref{prop:feq} and Theorem \ref{thm:rincl},
together with formula \eqref{eq:resx}.
\end{proof}

\subsection{Flat connections on curves}
\label{subsec:curves}
Our next objective is to extract from formulas \eqref{eq:feq}
and \eqref{eq:reseq} the irreducible decompositions 
of the varieties $\F(H^\hdot (X,\C), \g)$ and 
$\RR^1_1(H^\hdot (X,\C), \theta)$.  Before proceeding, 
let us first analyze in more detail the contribution
made by each admissible map.

\begin{lemma}
\label{lem:fas}
Let $S$ be a curve with $\chi(S)<0$, and let $\g=\sl_2$ or $\sol_2$.  
Then $\F(H^\hdot (S,\C), \g)$ is irreducible
and contains regular elements.
\end{lemma}

\begin{proof}
First suppose $S$ is non-compact. Let $n =\dim H^1 (S,\C)$;
since $\chi(S)<0$, we must have $n\ge 2$.  Since $H^2 (S,\C)=0$,
the cup-product map on $H^1 (S,\C)$ is trivial. Hence,
$\F(H^\hdot (S,\C), \g)=\g^n$, and the desired conclusions follow,
since $1<\dim \g$.

Next, suppose $S$ is compact. Let $g=\frac{1}{2}\dim H^1 (S,\C)$
be the genus of the curve; since $\chi(S)<0$, we must have $g\ge 2$.
The cup-product map, $\cup \colon \bigwedge^2 H^1 (S,\C) \to H^2 (S,\C)=\C$, 
defines a symplectic inner product $\langle \cdot, \cdot \rangle$
on $H^1(S,\C)$. Let us fix a symplectic basis $\{u_1, v_1, \dots, u_g, v_g\}$,
so that $\langle u_i, u_j \rangle= \langle v_i, v_j \rangle=0$ and 
$\langle u_i, v_j \rangle = \delta_{ij}$ for all $i, j$. 
Let $\Sp_{2g}$ be the linear, connected automorphism group of
$\langle \cdot, \cdot \rangle$, and let $U = \spn\{u_1, \dots, u_g\}$ be
the associated maximal isotropic subspace.

Set $r=\dim \g$.  As we saw in the proof of Proposition \ref{prop:feq},
the subvariety $\F(H^\hdot (S,\C), \g) \subseteq H^1 (S,\C) \otimes \g$ 
may be identified with the set of $r$-tuples $(\alpha_1, \dots, \alpha_r) \in (H^1(S,\C))^r$
spanning an isotropic subspace with respect to $\cup$. Consider
the regular map
\begin{equation}
\label{eq:orb}
\Phi \colon \Sp_{2g} \times U^r \rightarrow \F(H^\hdot (S,\C), \g), \quad
(s, \beta_1, \dots, \beta_r)\mapsto (s\beta_1, \dots, s\beta_r).
\end{equation}
As is well-known (see for instance \cite[Ch.~I]{MH}),
every maximal isotropic subspace of $ H^1(S,\C)$
is of the form $s(U)$, for some $s \in \Sp_{2g}$. Thus, the map
$\Phi$ is surjective, and the desired conclusions follow.
\end{proof}

Note that our proof actually gives an explicit description of the subvariety 
$ f^{!}(\F(H^\hdot (S,\C), \g)) \subseteq \F(H^\hdot (X,\C), \g)$.

\subsection{Irreducible decomposition}
\label{subsec:irred decomp}
We are now in a position to state and prove the main results of
this section.

\begin{theorem}
\label{thm:qpflat}
Let $X$ be a $1$-formal, quasi-projective manifold with $b_1(X)>0$,
and let $\g=\sl_2$ or $\sol_2$. Then, the irreducible decomposition of
$\F(H^{\hdot}(X,\C),\g)$ is given by
\begin{equation*}
\F(H^{\hdot}(X,\C), \g) =
\begin{cases}
P(H^1(X,\C) \times \g) & \text{if $\RR_1^1 (X)=\{0\}$,}\\
f^{!} (\F(H^\hdot (S,\C), \g)) 
& \text{if $\RR_1^1 (X)=H^1(X,\C)$,}\\
P(H^1(X,\C) \times \g) \cup \bigcup_{f\in \mathcal{E}_X}
f^{!} (\F(H^\hdot (S,\C), \g)) & \text{otherwise,}
\end{cases}
\end{equation*}
where, in the second case, $f\colon X\to S$ is an admissible map 
to a curve $S$ with $\chi(S) <0$, and likewise in the third case.
\end{theorem}

\begin{proof}
In all three cases, the stated equalities follow from Corollary \ref{cor:highres qp},
formula \eqref{eq:feq}.  The irreducibility of  $P(H^1(X,\C) \times \g)$
is clear, while the irreducibility of $f^{!} (\F(H^\hdot (S,\C), \g))$ follows 
from Lemma \ref{lem:fas}.
It remains to check that there are no redundancies in the last case.

Let $f\colon X\to S$ and $f'\colon X\to S'$ be two admissible maps.
In formula \eqref{eq:resx}, set $C=f^*(H^1(S,\C))$ and $C'=f'^*(H^1(S',\C))$,
and assume $C \nsubseteq C'$. Since 
$f^{!}(\F(H^\hdot (S,\C), \g))= \F(H^{\hdot}(X,\C), \g) \cap C \otimes \g$,
and similarly for $f'$, we infer that $f^{!}(\F(H^\hdot (S,\C), \g)) \nsubseteq
f'^{!}(\F(H^\hdot (S',\C), \g))$. 

Now, since $\RR^1_1(X)\ne H^1(X,\C)$, we must have $C \neq H^1(X,\C)$.
By a similar support argument, 
$P(H^1(X,\C) \times \g) \nsubseteq f^{!}(\F(H^\hdot (S,\C), \g))$. Finally,
recall from Lemma \ref{lem:fas}  that $ f^{!}(\F(H^\hdot (S,\C), \g))$ contains
regular elements.  Hence,  $ f^{!}(\F(H^\hdot (S,\C), \g)) \nsubseteq 
P(H^1(X,\C) \times \g)$, and this completes the proof.
\end{proof}

\begin{theorem}
\label{thm:qpres}
Let $X$ be a $1$-formal, quasi-projective manifold with $b_1 (X)>0$,
and let $\theta\colon \sl_2\to \gl(V)$ be a finite-dimensional representation.
Then, the decomposition into irreducible components of 
$\RR_1^1 (H^{\hdot}(X,\C),\theta) $ is given by
\begin{equation*}
\RR_1^1 (H^{\hdot}(X,\C),\theta) =
\begin{cases}
\Pi(H^{\hdot}(X,\C),\theta) & \text{if $\RR_1^1 (X)=\{0\}$,}\\
f^{!} (\F(H^\hdot (S,\C), \sl_2)) 
& \text{if $\RR_1^1 (X)=H^1(X,\C)$,}\\
\Pi(H^{\hdot}(X,\C),\theta) \cup \bigcup_{f\in \mathcal{E}_X}
f^{!} (\F(H^\hdot (S,\C), \sl_2)) & \text{otherwise,}
\end{cases}
\end{equation*}
where, in the second case, $f\colon X\to S$ is an admissible map 
to a curve $S$ with $\chi(S) <0$, and likewise in the third case.
\end{theorem}

\begin{proof}
In all three cases, the stated equalities follow from Corollary \ref{cor:highres qp},
formula \eqref{eq:reseq}.  Furthermore, the irreducibility of $\Pi(H^{\hdot}(X,\C),\theta)$ 
follows from Lemma \ref{lem:det}, while the irreducibility of $f^{!} (\F(H^\hdot (S,\C), \sl_2))$
follows from Lemma \ref{lem:fas}.

It remains to check that there are no redundancies in the last equality.
In view of Theorem \ref{thm:qpflat}, it is enough to verify
that $\Pi(H^{\hdot}(X,\C), \theta) \nsubseteq f^{!}(\F(H^\hdot (S,\C), \sl_2))$,
where $\theta$ is the inclusion $\sl_2 \inj \gl_2$. 

By our assumption on $\RR_1^1(X)$, there is an element 
$\eta \in H^1(X,\C) \setminus f^* (H^1(S,\C))$. 
Pick a non-zero matrix $g\in \sl_2$ such that 
$\det (\theta(g))=0$, and set
$\omega=\eta \otimes g \in \Pi(H^{\hdot}(X,\C), \theta)$.
Clearly, $\omega \notin f^{!}(\F(H^\hdot (S,\C), \sl_2))$,
and so we are done.
\end{proof}

\subsection{Discussion}
\label{subsec:notpull}
Let us compare the rank $1$ irreducible decomposition \eqref{eq:resx}
with the decomposition from Theorem \ref{thm:qpres}.  To that end,
let $A=(H^\hdot (X,\C), d=0)$ be the formal model of a $1$-formal,
quasi-projective manifold with $b_1(X)>0$, and let $V_{\theta}$ be
a finite-dimensional $\sl_2$-module.  Assuming $\RR_1^1 (X)$ is
neither $\{0\}$, nor $A^1$, we see that the irreducible components
of $\RR_1^1(A, \theta)$ either arise by pullback along admissible
maps, or are equal to $\Pi(A, \theta)$. When $\theta=\id_\C$,
only the first possibility occurs.

Both in the rank $1$ case and the nonabelian situation, geometric irreducible
components (to be denoted $\W_f$) have the following property: There is an 
epimorphism of Lie algebras, $q \colon \h(A) \twoheadrightarrow \kk$, such that
$W_f=q^*(\Hom_{\Lie}(\kk, \sl_2))$. This assertion may be verified as follows.
First, note that $\W_f=f^{!}(\F(A(S), \sl_2))$, by construction, where
$A(S)=(H^\hdot (S,\C), d=0)$ and $f^* \colon A^{\hdot}(S) \inj A^\hdot$
is injective.  Finally, we may use the natural identification from
Proposition \ref{prop:flathol}, and put $\kk=\h(A(S))$ and $q= \h(f^*)$.

Next, we note that the irreducible component $W =  \Pi (A, \theta)$
has a different qualitative behavior from this pullback point of view. Indeed, let
$A$ be an arbitrary $1$-finite $\cdga$ with $H^1 (A) \neq 0$ and let
$\theta \colon \sl_2 \to \gl_2$ be the inclusion map. We claim that
there is no epimorphism $q \colon \h(A) \surj \kk$ such that
$W=q^*(\Hom_{\Lie}(\kk, \sl_2))$.

Indeed, using the definition of the holonomy Lie algebra we may write
$\kk=\L(A_1) \slash R$, for some ideal $R$. Let
$\init(R) \subseteq A_1$ be the vector subspace generated by
the degree $1$ components of  all elements $r \in R$.  If
$\init(R) \neq  A_1$, pick a basis $\{x_1, \dots, x_n\}$ of $A_1$
such that $\{x_1, \dots, x_m\}$, for some $m<n$, form a basis of
$\init(R)$. Next, take a matrix $g \in \sl_2$ with $\det(\theta(g))\ne 0$, and define
a morphism $\rho \colon \kk \to \sl_2$ by setting $\rho(x_i)=0$ for $i \leq m$ and
$\rho(x_i)=g$, for $i > m$. Clearly, $q^*(\rho) \notin W$. If $\init(R)=A_1$,
the equality $W=q^*(\Hom_{\Lie}(\kk, \sl_2))$ would imply $W=0$, by
Proposition \ref{prop:flathol}, thereby contradicting
the assumption that $H^1 (A) \neq 0$.

\subsection{Positive weights}
\label{subset:weights}
We now come back to the general quasi-projective case, as described
in Proposition \ref{prop:flatres qp}. The global results of Corlette
and Simpson from \cite{CS}, together with the local Theorem \ref{thm:b1}, 
suggest that the analysis of the parameter space for flat, $\g$-valued 
connections on the Gysin model $A^{\hdot}$ is much more complicated 
than in the $1$-formal case, even when $\g=\sl_2$.

A useful feature of Gysin models is the existence of a weight decomposition
with weights $1$ and $2$ in degree $1$ (cf.~Morgan \cite{Mo}). This means
we have a direct sum decomposition, $A^\hdot= \bigoplus_{j \in \Z} A ^\hdot_j$,
such that $A ^\hdot_j \cdot A ^\hdot_k \subseteq A ^\hdot_{j+k}$,
$dA^\hdot_j \subseteq A ^\hdot_j$ and $A ^1_j=0$ for $j \neq 1,2$.
Our next result implies that inclusion \eqref{eq:inclf} becomes an equality for
$\g=\sl_2$ or $\sol_2$, in the case when $X$ is quasi-projective
and $\RR_1^1(X)=\{0\}$. Indeed, Theorem \ref{thm:tcone}
guarantees that  $\RR_1^1 (A)=\{0\}$, hence the right hand side
in \eqref{eq:inclf} reduces to $\F^1(A, \g)$.

\begin{prop}
\label{prop:nopencils}
Let $A^\hdot$ be a connected \cdga~admitting a weight decomposition with weights
$1$ and $2$ in degree $1$, and suppose $\RR^1_1(H^\hdot(A)) \subseteq \{0\}$.
Then $\F(A, \g)=\F^1(A, \g)$, for $\g=\sl_2$ or $\sol_2$.
\end{prop}

\begin{proof}
We first treat the case $\g=\sl_2$. An arbitrary element $\omega \in \F(A, \g)$
can then be written as $\alpha \otimes X + \beta \otimes Y + \gamma \otimes H$, with
$\alpha, \beta, \gamma \in A^1$. It is readily checked that the Maurer--Cartan
equation \eqref{eq:flat} is now equivalent to the following system of equations:
\begin{equation}
\label{eq:qpmc}
d \gamma + \alpha \beta =0, \; d \alpha+ 2\gamma \alpha=0, \; d \beta-2 \gamma \beta=0
\end{equation}

We have to show that $\rank \{\alpha, \beta,\gamma\} \leq 1$.
To this end, let $\alpha =\alpha_1+\alpha_2$, $\beta=\beta_1+\beta_2$, and
$\gamma=\gamma_1 + \gamma_2$ be the respective weight decompositions.
The weight $1$ part of \eqref{eq:qpmc} is
\begin{equation}
\label{eq:qpmc1}
d \alpha_1=d \beta_1=d \gamma_1=0
\end{equation}

The weight $2$ component is
\begin{equation}
\label{eq:qpmc2}
d \gamma_2 + \alpha_1 \beta_1 = d \alpha_2+2 \gamma_1 \alpha_1=
d \beta_2- 2 \gamma_1\beta_1=0.
\end{equation}

For the weight $3$ component we find that
$\alpha_1 \beta_2- \beta_1  \alpha_2 =  \gamma_1 \alpha_2 - \alpha_1 \gamma_2 =
\gamma_1\beta_2 - \beta_1 \gamma_2=0$,
and, finally, the weight $4$ component gives
\begin{equation}
\label{eq:qpmc4}
 \alpha_2 \beta_2=  \gamma_2 \alpha_2= \gamma_2 \beta_2 =0.
\end{equation}

Using our hypothesis on $\RR^1_1(H^\hdot (A))$, we infer
from \eqref{eq:qpmc1} and  \eqref{eq:qpmc2} that
$\rank \{\alpha_1, \;\beta_1, \; \gamma_1\} \leq 1$ and
$d \alpha_2 = d \beta_2=  d \gamma_2 =0$. Using \eqref{eq:qpmc4},
the same argument implies that $\rank \{\alpha_2, \;\beta_2, \; \gamma_2\} \leq 1$.
If $\rank \{\alpha_1, \;\beta_1, \; \gamma_1\}=0$ (respectively,
$\rank \{\alpha_2, \;\beta_2, \; \gamma_2\}=0 $), we are done.
Thus, we may assume that $\alpha_1=a \omega_1$, $\beta_1=b \omega_1$,
$\gamma_1= c \omega_1$, with $0 \neq \omega_1 \in A^1_1$, $d \omega_1=0$
and $(a,b,c) \neq 0$. Similarly, we may write $\alpha_2=a' \omega_2$,
$\beta_2=b' \omega_2$, $\gamma_2=c' \omega_2$, with
$0 \neq \omega_2 \in A_2^1$, $d \omega_2=0$, and $(a',b',c') \neq 0$.

Using \eqref{eq:qpmc}, we deduce that either $\omega_1 \omega_2=0$,
or $ab'-ba'=ca'-ac'=cb'-bc'=0$. In the first case, we must have
$\rank\{\omega_1, \; \omega_2\}=2$, contradicting the
assumption that $\RR^1_1(H^\hdot (A)) \subseteq \{0\}$. In the
second case, $(a',b',c')= \lambda (a,b,c)$, which implies
$\omega \in \F^1(A, \g)$ and completes the proof in the
case when $\g=\sl_2$.

When $\g=\sol_2$, an element $\omega \in \F(A, \g)$ may be written as
$\omega= \alpha \otimes h + \beta \otimes x$, with $\alpha, \beta \in A^1$.
In this case, the Maurer--Cartan equation \eqref{eq:flat} is equivalent to
the system of equations $d \alpha=0$ and $d \beta+ \alpha \beta =0$.
An argument as before shows that $\rank\{\alpha, \;\beta\} \leq 1$,
thereby completing the proof.
\end{proof}

We obtain as a corollary another class of examples where higher rank embedded
jump loci are controlled by information on the corresponding rank $1$ loci.

\begin{corollary}
\label{cor=finalqp}
Let $X$ be a quasi-projective manifold with $b_1 (X)>0$, 
and let $\theta\colon \g \to \gl(V)$ be a finite-dimensional representation,
where $\g=\sl_2$ or $\sol_2$. Assume $\RR_1^1(X)=\{0\}$. Then
the inclusions from Proposition \ref{prop:flatres qp} become equalities, 
for any Gysin model $A$ of $X$. More precisely,
\begin{equation}
\label{eq:nopen}
\F(A, \g) = \F^1(A, \g) \quad \text{and} \quad \RR^1_1(A, \theta)= \Pi(A, \theta).
\end{equation}
\end{corollary}

\begin{proof}
By Theorem \ref{thm:tcone}, $\RR_1^1 (A)=\{0\}$; thus, 
by formula \eqref{eq:hpencils},  $\mathcal{E}_X= \emptyset$. 
Hence, all we have to do is to prove equalities \eqref{eq:nopen}.

The first equality follows from Proposition \ref{prop:nopencils}, 
while the inclusion $\Pi(A, \theta) \subseteq \RR^1_1(A, \theta)$ 
follows from Corollary \ref{cor:vdet}.  To check the reverse inclusion, 
pick $0\neq \omega =\eta \otimes g \in \RR^1_1(A, \theta)$.
Again by Corollary \ref{cor:vdet}, we find an eigenvalue $\lambda$ 
of $\theta(g)$ such that $\lambda \eta \in \RR^1_1(A)$. Hence, 
$\lambda =0$, and therefore $\omega \in \Pi(A, \theta)$.
\end{proof}

For a detailed study of algebras having vanishing resonance, we refer 
to \cite{PS-crelle}. Finally, it should be noted that the classical, depth $1$ 
resonance varieties  behave well under finite products 
and coproducts, see \cite[\S 13]{PS-plms}. Using some of the machinery 
developed in this paper, we show in \cite{PS-prod} that analogous product 
and coproduct formulas hold in the non-abelian setting. 

\begin{ack}
Some of this work was done while S.~Papadima and A.~Suciu visited 
the Max Planck Institute for Mathematics in Bonn in April--May 2012. 
The work was continued while A.S. visited the Institute of 
Mathematics of the Romanian Academy in June, 2012 and 
June, 2013. The work was completed while A.S. visited 
again MPIM Bonn in September--October 2013.  Thanks are 
due to both institutions for their hospitality, support, and excellent 
research atmosphere.   We also thank the two referees for their 
useful remarks and suggestions.
\end{ack}

\newcommand{\arxiv}[1]
{\texttt{\href{http://arxiv.org/abs/#1}{arxiv:#1}}}
\renewcommand{\MR}[1]
{\href{http://www.ams.org/mathscinet-getitem?mr=#1}{MR#1}}
\newcommand{\doi}[1]
{\texttt{\href{http://dx.doi.org/#1}{doi:#1}}}

\end{document}